\documentclass[11pt]{article}

\usepackage[square,numbers]{natbib}
\usepackage{amsmath,amssymb,amsthm,mathtools}
\usepackage{bm}
\usepackage{geometry}
\usepackage{microtype}
\usepackage{xcolor}
\usepackage{hyperref}
\usepackage[nameinlink,noabbrev]{cleveref}
\usepackage{aliascnt}
\usepackage{authblk}

\newcommand{\OT}{\mathsf {OT}_{\varepsilon}}

\newcommand{\ind}{\mathbf 1}

\hypersetup{
  colorlinks=true,
  linkcolor=red!65!black,
  citecolor=blue!65!black,
  urlcolor=blue!65!black
}

\usepackage{amsbsy}
\usepackage{mathrsfs}
\usepackage{amsmath}
\usepackage{amsfonts}
\usepackage{amssymb}
\usepackage{graphicx}
\usepackage{times}
\usepackage{comment}
\usepackage{dsfont}
\usepackage{bm}

\usepackage{color}
\usepackage{fancyhdr}
\allowdisplaybreaks[4]

\usepackage{enumerate}
\usepackage{enumitem}

\newcommand{\cB}{\mathcal{B}}\newcommand{\cC}{\mathcal{C}}
\newcommand{\cD}{\mathcal{D}}\newcommand{\cE}{\mathcal{E}}\newcommand{\cF}{\mathcal{F}}
\newcommand{\cH}{\mathcal{H}}

\newcommand{\cP}{\mathcal{P}}
\newcommand{\cS}{\mathcal{S}}
\newcommand{\cX}{\mathcal{X}}
\newcommand{\cY}{\mathcal{Y}}\newcommand{\cZ}{\mathcal{Z}}  
\newcommand{\bE}{\mathbb{E}}

\newcommand{\bR}{\mathbb{R}}

\DeclareMathOperator{\tr}{tr} 

\theoremstyle{plain}
\newtheorem{theorem}{Theorem}[section]

\newaliascnt{proposition}{theorem}
\newtheorem{proposition}[proposition]{Proposition}
\aliascntresetthe{proposition}

\newaliascnt{lemma}{theorem}
\newtheorem{lemma}[lemma]{Lemma}
\aliascntresetthe{lemma}

\newaliascnt{corollary}{theorem}
\newtheorem{corollary}[corollary]{Corollary}
\aliascntresetthe{corollary}

\theoremstyle{definition}

\newaliascnt{definition}{theorem}
\newtheorem{definition}[definition]{Definition}
\aliascntresetthe{definition}

\newaliascnt{setting}{theorem}

\aliascntresetthe{setting}

\newaliascnt{remark}{theorem}
\newtheorem{remark}[remark]{Remark}
\aliascntresetthe{remark}

\newaliascnt{example}{theorem}

\aliascntresetthe{example}

\newaliascnt{assumption}{theorem}
\newtheorem{assumption}[assumption]{Assumption}
\aliascntresetthe{assumption}
\crefname{equation}{}{}
\crefname{theorem}{Theorem}{Theorems}
\crefname{proposition}{Proposition}{Propositions}
\crefname{lemma}{Lemma}{Lemmas}
\crefname{corollary}{Corollary}{Corollaries}
\crefname{definition}{Definition}{Definitions}
\crefname{setting}{Setting}{Settings}
\crefname{remark}{Remark}{Remarks}
\crefname{example}{Example}{Examples}
\crefname{assumption}{Assumption}{Assumptions}

\Crefname{equation}{}{}
\Crefname{theorem}{Theorem}{Theorems}
\Crefname{proposition}{Proposition}{Propositions}
\Crefname{lemma}{Lemma}{Lemmas}
\Crefname{corollary}{Corollary}{Corollaries}
\Crefname{definition}{Definition}{Definitions}
\Crefname{setting}{Setting}{Settings}
\Crefname{remark}{Remark}{Remarks}
\Crefname{example}{Example}{Examples}
\Crefname{assumption}{Assumption}{Assumptions}

\title{Quantitative Stability and Sample Complexity for Regularized
Gromov--Wasserstein Distance}

\author[1,*]{Pengtao Li}
\author[1,2,3,*]{Xiaohui Chen}

\affil[1]{Department of Mathematics,
University of Southern California}

\affil[2]{Thomas Lord Department of Computer Science,
University of Southern California}

\affil[3]{Computing and Mathematical Sciences,
California Institute of Technology}

\affil[*]{Email: 
\texttt{\{pengtaol, xiaohuic\}@usc.edu}}

\date{}

\begin{document}
\maketitle

\begin{abstract}
The Gromov–Wasserstein (GW) distance provides a principled framework for comparing two probability distributions supported on different metric spaces by matching their individual intrinsic geometric structures. In practice, the GW distance is estimated from finite samples and often incorporates regularization for statistical and computational tractability. This motivates a quantitative understanding of the stability and estimation error of regularized GW distances. In this work, we study a general family of divergence regularized GW distances, including entropy and Tsallis regularizations over the full range of powers, from a nonasymptotic perspective. Specifically, we establish Lipschitz continuity of the value of regularized GW and H{\"o}lder continuity of its optimal coupling under Wasserstein perturbations of the input probability measures. Moreover, we derive the parametric rate of convergence for the empirical regularized $(2,2)$-GW distance on compactly supported Euclidean spaces.
\end{abstract}

\section{Introduction}\label{sec:introduction}
The Gromov--Wasserstein (GW) distance \cite{Mem11,Stu12}
quantifies discrepancy between probability distributions supported on different
metric spaces through their internal distance structures. 
Given two metric measure spaces
$(\mathcal{X},d_{\mathcal{X}},\mu)$ and
$(\mathcal{Y},d_{\mathcal{Y}},\nu)$, and the kernel
$
\Gamma_{r,s}\bigl((x,y),(x',y')\bigr)
:=
\left|
d_{\mathcal{X}}(x,x')^s-d_{\mathcal{Y}}(y,y')^s
\right|^r,
$
the associated $(r,s)$-GW distance is defined by 
\begin{equation*}
\mathsf{GW}_{r,s}(\mu,\nu)
:=\inf_{\pi\in\Pi(\mu,\nu)}
\biggl(
\mathcal{E}_{\Gamma_{r,s}}(\pi)
:=\int_{\mathcal{X}\times\mathcal{Y}}\int_{\mathcal{X}\times\mathcal{Y}}
\Gamma_{r,s}
\,d\,(\pi\otimes\pi)
\biggr)^{1/r},
\end{equation*}
where $\Pi(\mu,\nu)$ denotes the set of probability measures (i.e., couplings) on $\mathcal{X}\times\mathcal{Y}$ with marginals $\mu$ and $\nu$.
When the context is clear, we write $\Gamma_{r,s}$ as $\Gamma$.
Rooted in the theory of optimal transport (OT)
\cite{Vil09,villani2021topics}, the Gromov--Wasserstein formulation
provides an $L^p$ relaxation of the
Gromov--Hausdorff distance \cite{Mem11}.
A key advantage of this formulation is that it does not require
a prescribed transportation cost between points in different spaces.
This is particularly useful for heterogeneous or unaligned datasets,
for which a meaningful cross-domain cost may be difficult to define
\cite{AlvarezMelisJaakkola2018}. %[Section~3.1]
Instead, GW compares within-space distances and seeks a coupling
of the underlying probability measures that preserves these
distances as closely as possible
\cite{Mem11,peyre2016gromov}.
This flexibility has enabled applications in shape correspondence
and graph matching
\cite{Mem09,XLZD19, chen_yao_tang_chen2023gromov_graph},
language alignment \cite{AlvarezMelisJaakkola2018},
generative modeling \cite{BunneEtAl2019},
and single-cell multi-omics alignment \cite{demetci2022scot}.

Entropic regularization offers a practical approach to alleviating the NP-hard computational burden for exact   GW distance, 
by enabling iterative algorithms based on efficiently solvable entropic OT subproblems
\cite{peyre2016gromov,rioux_kengo2024entropicGW}.
For $\varepsilon>0$, the entropically regularized GW functional is
defined by
\begin{equation}\label{eq:entropy_regularized-gw-value}
S_{r,s,\varepsilon}^{\text{ent}}(\mu,\nu)
:=
\inf_{\pi\in\Pi(\mu,\nu)}
\left\{
\mathcal E_{\Gamma_{r,s}}(\pi)
+
\varepsilon\,\mathsf D_{\mathrm{KL}}
\bigl(\pi,\mu\otimes\nu\bigr)
\right\},
\end{equation}
where $\mathsf D_{\mathrm{KL}}$ denotes the entropy (a.k.a. Kullback--Leibler) divergence. 
However, entropic regularization inherently entails diffuse couplings: every entropic optimizer has a strictly positive density with respect to the product of the marginals, and hence has full support
on the product of their supports \cite{nutz2021introduction, Chewi_Niles_Rigollet_statistical_OT_2025}.
This contrasts with the concentrated structure of optimal couplings
available in several unregularized GW settings. 
In particular, under suitable assumptions on the marginals, unregularized Euclidean $(2,2)$-GW admits an optimal coupling concentrated on the union of at most two graphs \cite{dumont2025mongeGWexistence}. In several structured settings, even a single graph suffices. For instances, Monge optimizers are known for rotationally invariant absolutely continuous measures on the same Euclidean space \cite{Stu12}, for uniform measures on Euclidean spheres \cite{arya2026gromovsphere}, and for Gaussian measures with proportional covariance spectra \cite{delon2022gromovGaussian}.
Thus, while entropy facilitates computation, its tendency to spread mass can obscure sharp correspondences and lead to blurring in OT-based applications. 
\begin{comment}
This trade-off is particularly relevant for GW because of the generally nonconvex structure of its quadratic distortion. 
For entropic GW, known sufficient conditions for strict convexity of the matrix variational objective, and hence optimizer uniqueness, 
require $\varepsilon$ to exceed problem-dependent thresholds \cite{rioux_kengo2024entropicGW,RiouxGoldfeldKatoLimitLaws2024}.
However, stronger regularization favors the independent coupling $\mu\otimes\nu$, 
so meeting these thresholds  may conflict with the goal of preserving sharply localized correspondences and exacerbate entropic blurring effect.
\end{comment}

In regularized OT formulations,
alternative penalties, including Tsallis divergences, 
allow the optimal density to vanish and thus permit sparse couplings
\cite{Nutz.2024,BayraktarEcksteinZhang_stability,pengtao_alberto_xiaohui_2026centrallimittheoremsbootstrap}.
This motivates replacing the entropic penalty in
\eqref{eq:entropy_regularized-gw-value} by a general divergence and investigating the resulting regularized GW problem.
Thus, in this paper we consider, for $\varepsilon>0$, regularized $(r,s)$-GW functional
\begin{equation}\label{eq:regularized-gw-value}
S_{r,s,\varepsilon}(\mu,\nu)
:=
\inf_{\pi\in\Pi(\mu,\nu)}
\Bigl\{
\cF_\varepsilon^{\mu,\nu}(\pi)
:=
\cE_{\Gamma_{r,s}}(\pi)+\varepsilon\,
\mathsf{D}_\varphi(\pi,\mu\otimes\nu)
\Bigr\},
\end{equation}
where for probability measures $Q\ll R$ on a common space, the $\varphi$-divergence is defined as
\[
\mathsf{D}_\varphi(Q,R)
=
\int\varphi\Bigl(\frac{dQ}{dR}\Bigr)dR,
\]
and we set $\mathsf{D}_\varphi(Q,R) := +\infty$ if $Q\not\ll R$. The function $\varphi:[0,\infty)\to\bR$ is continuous and strictly convex, with $\varphi(1)=0$ and $\lim_{x\to\infty}\varphi(x)/x=+\infty$.
This formulation includes entropic regularization as well as Tsallis regularization as special cases.
%This work addresses the stability and sample complexity of the regularized Gromov-Wasserstein distance.

\subsection{Overview of the main results}

We study $\varphi$-divergence regularized GW distance from a non-asymptotic perspective, focusing on two aspects:   
quantitative stability of optimal couplings under perturbations of the marginals,
and finite-sample bounds for estimating the regularized functional from independent samples. Throughout, we impose the following assumption on $\varphi$, which is satisfied by the entropy regularizer $\varphi(x)=x\log x$, with $0 \cdot \log0:=0$, and the Tsallis regularizers $\varphi_\alpha(x)=(x^\alpha-1)/\alpha$ for every $\alpha>1$.
\begin{assumption}[Divergence regularizer]
\label{assumption:regularizaion}
The function $\varphi$ belongs to $\cC^2((0,\infty))$, satisfies
$\varphi''(s)>0$ for $s>0$, and there exist
$\beta\ge0$ and $\lambda_0,\lambda_1,\lambda_2\ge0$
such that
\begin{equation}\label{eq:Aphi}
\frac{1}{\varphi''(s)}
\le\lambda_0s^{-\beta}+\lambda_1+\lambda_2s,
\qquad s>0.
\end{equation}
\end{assumption}
For Kullback-Leibler divergence, \Cref{assumption:regularizaion} holds with
$(\beta,\lambda_0,\lambda_1,\lambda_2)=(0,0,0,1).$
The inclusion of the parameter $\beta$ further allows us to cover Tsallis
regularization for the full range $\alpha>1$,
thereby including the regime $\alpha>2$ that is excluded by the condition
$
1/\varphi''(s)
\le \lambda_1+\lambda_2 s
$
imposed in \cite{BayraktarEcksteinZhang_stability}. 
Indeed, if $1<\alpha\le2$, then \Cref{assumption:regularizaion} is satisfied with
$
(\beta,\lambda_0,\lambda_1,\lambda_2)
=
\Bigl(0,0,1,\frac{2-\alpha}{\alpha-1}\Bigr),
$
by concavity of the map $s\mapsto s^{2-\alpha}$. If $\alpha>2$, it is satisfied with
$
(\beta,\lambda_0,\lambda_1,\lambda_2)
=
\bigl(\alpha-2,\frac{1}{\alpha-1},0,0\bigr).
$
Throughout the paper, we assume without loss of generality that $\varphi'(1)=1$, since replacing $\varphi(x)$ by $\varphi(x)+(1-\varphi'(1))(x-1)$ leaves the divergence between probability measures unchanged. 

\begin{comment}
For Tsallis regularization, $\varphi_\alpha(s):=\frac{s^\alpha-1}{\alpha},\alpha>1$ for $s\geq0$.
Then $\frac1{\varphi_\alpha''(s)}
=\frac{s^{2-\alpha}}{\alpha-1}$.
If $1<\alpha\le2$, concavity of $s\mapsto s^{2-\alpha}$
gives $s^{2-\alpha}
\le
(\alpha-1)+(2-\alpha)s$
for $s>0$
and hence
$\frac1{\varphi_\alpha''(s)}
\le
1+\frac{2-\alpha}{\alpha-1}s$.

If $\alpha>2$, then
$
\frac1{\varphi_\alpha''(s)}
=
\frac1{\alpha-1}s^{-(\alpha-2)}
$
and then  \Cref{assumption:regularizaion} holds with
$$ (\beta,\lambda_0,\lambda_1,\lambda_2)
=
\Bigl(
\alpha-2,\frac1{\alpha-1},0,0
\Bigr).$$
\end{comment}

\subsubsection{Quantitative stability of optimizers}
Our first main result quantifies the sensitivity of optimal
couplings to perturbations of the marginal distributions measured in Wasserstein distances. Specifically, we establish the continuity of the map
\begin{align}\label{eqn:introduction_mapping}
(\mu,\nu)
\mapsto
\pi^*
:=
\arg\min_{\pi\in\Pi(\mu,\nu)}
\cE_{\Gamma}(\pi)+\varepsilon\,
\mathsf{D}_\varphi(\pi,\mu\otimes\nu)
\end{align}
in \Cref{theorem:DGW-stability}. When both the marginals and the corresponding optimizers are measured in the
$q$-Wasserstein distance, we prove that the map in \eqref{eqn:introduction_mapping} is
H\"older continuous of order $1/[q(\beta+2)]$.

Our proof builds on the shadow-based strategy developed in \cite{eckstein2022quantitative,BayraktarEcksteinZhang_stability} for regularized OT (ROT) problems. 
A shadow (\Cref{definition:shadow}) transfers a coupling to the perturbed marginal constraints while preserving its Wasserstein displacement without increasing its divergence relative (cf. \Cref{proposition_shadow_property}). 
% In contrast to these results on ROT, a key distinction concerns the well-definedness of the mapping (4) under Gromov-Wasserstein setting. For fixed marginals (μ, ν), the ROT objective
However, a key distinction of our Gromov-Wasserstein setting from ROT concerns the well-definedness of the mapping \eqref{eqn:introduction_mapping}.
To elaborate, the ROT objective
\[
\pi\longmapsto
\int c\,d\pi
+\varepsilon\,\mathsf D_\varphi(\pi,\mu\otimes\nu)
\]
for fixed marginals $(\mu,\nu)$, is strictly convex on its domain. %: the transport term is linear, while the divergence is strictly convex. 
%in its first argument when $\varphi$ is strictly convex. 
Since the feasible set $\Pi(\mu,\nu)$ is convex,
any finite minimizer is therefore unique
\cite[Lemma~2.1]{BayraktarEcksteinZhang_stability}.
Unfortunately, this claim does not extend directly to regularized GW.
Indeed, the distortion 
\[
\pi\mapsto
\cE_\Gamma(\pi)
=\int_{\mathcal{X}\times\mathcal{Y}}
\int_{\mathcal{X}\times\mathcal{Y}}
\left|
d_{\mathcal{X}}(x,x')^s-d_{\mathcal{Y}}(y,y')^s
\right|^r
\,d\pi(x,y) d\pi(x',y')
\]
is quadratic in the coupling $\pi$ and
may not be convex. For $\pi_0,\pi_1\in\Pi(\mu,\nu)$
and $\pi_t=(1-t)\pi_0+t\pi_1$, its interpolation identity is
\[
\cE_\Gamma(\pi_t)
=
(1-t)\cE_\Gamma(\pi_0)
+t\cE_\Gamma(\pi_1)
-t(1-t)\cE_\Gamma(\pi_1-\pi_0),
\qquad 0<t<1.
\]
The signed quadratic term $\cE_\Gamma(\pi_1-\pi_0)$
can be negative, so the strict convexity of the divergence
penalty does not by itself guarantee convexity or uniqueness
for the regularized GW problem. Consequently, without
additional structure, uniqueness and thus stability of the mapping in \eqref{eqn:introduction_mapping} are irrelevant.

We address this difficulty in \Cref{assumption:Cgamma}, through curvature domination relative to a minimizer, to  control the negative quadratic remainder
and ensure uniqueness of the optimizer (\Cref{corollary:GW-uniqueness}), as a sufficient condition. 
Such sufficiently strong regularization requirement also appears in related entropic GW results \cite{rioux_kengo2024entropicGW, RiouxMarksPasseggeriGoldfeld2026, RiouxGoldfeldKatoLimitLaws2024}. 
%Moment-dependent and spectral conditions for convexity of auxiliary variational objectives are established in \cite[Theorem~6]{rioux_kengo2024entropicGW} and \cite[Proposition~7]{RiouxMarksPasseggeriGoldfeld2026}, respectively. The resulting uniqueness is also used in \cite[Theorem~9]{RiouxGoldfeldKatoLimitLaws2024} to obtain Gaussian limits and bootstrap consistency for empirical entropic GW values.

This type of result regarding H\"older continuity of optimal couplings in Wasserstein distances is first obtained by \cite{eckstein2022quantitative} of order $1/2q$ for entropic OT, under exponential moment assumptions. 
It was subsequently extended in \cite{BayraktarEcksteinZhang_stability} to general divergence regularization under weaker moment assumptions. 
The latter framework nevertheless excludes Tsallis regularization with power $\alpha>2$, whereas our regularity condition accommodates the full range $\alpha>1$.
The concurrent H\"older exponent $1/[q(\beta+2)]$ matches $1/2q$ therein for entropy and Tsallis powers $1<\alpha\le2$, for which $\beta=0$. 
When $\alpha>2$, taking $\beta=\alpha-2$ yields the exponent $1/\alpha q$.

\subsubsection{Sample complexity}

We next study a statistical instance of marginal perturbation, in which the population marginals are replaced by empirical measures constructed from independent samples. 
Our goal is to quantify the resulting rate of convergence for the regularized GW value. 
Here we specialize to regularized $(2,2)$-GW problem and compactly supported distributions $\mu\in\mathcal P(\mathbb R^{d_x})$ and $\nu\in\mathcal P(\mathbb R^{d_y})$.
Let $\mu_n$ and $\nu_n$ be the empirical measures of two independent samples, each consisting of $n$ independent and identically distributed (i.i.d.) observations from its respective marginal. 
For every fixed $\varepsilon>0$, Theorem~\ref{thm:parametric_sample_complexity} yields
\begin{align}
\mathbb E\left|
S_{2,2,\varepsilon}(\mu_n,\nu_n)-S_{2,2,\varepsilon}(\mu,\nu)
\right|^2
\le C n^{-1},
\label{eq:intro-sample-l2}
\end{align}
for some constant $C>0$. We highlight that these bounds hold for entropy regularization and all Tsallis regularization with powers
$\alpha>1$. 
%Unlike \Cref{theorem:DGW-stability}, neither curvature domination nor uniqueness of GW minimizers is required here.
Next, we explain our proof structure of \eqref{eq:intro-sample-l2}, where our starting point of the proof is a variational representation connecting regularized GW to ROT.
Extending the matrix reformulation of \cite{zhang_ziv2024gromov}, Theorem~\ref{thm:gw-duality} shows that
\begin{align}\label{eq:intro-dual}
S_{2,2,\varepsilon}(\mu,\nu)-\mathsf{GW}^1(\mu,\nu)
=
\min_{A\in \mathcal D}
\left\{
32\|A\|_F^2+\OT^{c_A}(\mu,\nu)
\right\},
\end{align}
where $\mathsf{GW}^1$ is an explicit marginal-moment term \eqref{eqn:def_GW1},
$\mathcal D\subset\mathbb R^{d_x\times d_y}$ is a compact matrix
domain, and
$c_A(x,y)=-4\|x\|^2\|y\|^2-32x^\top Ay$. 
Duality of the inner ROT problem then makes empirical-process methods available. 
Unlike the usual argument in entropic setting \cite{ZivGoldfeldKengoKatoGabrielRiouxRitwikSadhu, zhang_ziv2024gromov} to leverage sufficient smoothness of the dual potentials, general $\varphi$-divergence regularization may produce vanishing optimal densities while providing insufficient  smoothness \cite{AlbertoGonzalez-SanzEcksteinStephanNutzMarcel2025Lpweaklimit,pengtao_alberto_xiaohui_2026centrallimittheoremsbootstrap}. In particular, our main analytical argument is to control the family of
population optimal dual integrands uniformly over cost functions family $c_A, A\in \mathcal D$, without relying on higher-order smoothness.
Lemma~\ref{lem:covering} establishes the covering bound 
\[
\log N(\eta,\mathcal H,\|\cdot\|_\infty)
\le
C\log( 1+ C/\eta),
\qquad \eta> 0,
\]
for the uniformly bounded class $\mathcal H$ of population optimal dual integrands. Here, $N(\eta,\cH ,\|\cdot\|_\infty)$ denotes the covering number of $\cH$ under $\|\cdot\|_\infty$-balls of radius $\eta$.
The proof exploits the curvature of the conjugate $\psi$ on regions where the associated optimal density is bounded away from zero. 
Specifically, we construct a fixed finite partition of one marginal support and prove that the matrix parameter, together with the averages of one potential over the partition cells, controls the dual potentials in the uniform norm.
This yields a finite-dimensional H\"older parametrization of $\mathcal H$, without requiring higher-order smoothness.
Then a chaining argument yields
\[
\mathbb E\sup_{h\in\mathcal H}
\left|
\bigl(\mu_n\otimes\nu_n-\mu\otimes\nu\bigr)h
\right|
\le Cn^{-1/2}.
\]
Combining with recent progress in ROT bias estimate  \cite{alberto.gonzalez-marcel.nutz_austin.stromme2026}, recalled in
Proposition~\ref{proposition:ROT_bias_parametric_bound}, we obtain the mean absolute error bound
\begin{align}
\mathbb E\left|
S_{2,2,\varepsilon}(\mu_n,\nu_n)-S_{2,2,\varepsilon}(\mu,\nu)
\right|
\le C n^{-1/2},
\label{eq:intro-sample-l1}
\end{align}
Finally, Efron--Stein
inequality gives a variance bound of order $n^{-1}$,
which, combined with \eqref{eq:intro-sample-l1}, proves
\eqref{eq:intro-sample-l2}.

\subsection{Related literature}
\label{subsec:related-literature}

\paragraph{Exact GW problems.}
Statistical analysis of unregularized GW has been substantially progressed in recent years.
A variational representation of $(2,2)$-GW problem through a
finite-dimensional family of ordinary OT problems, together with empirical
convergence rates, is established in \cite{zhang_ziv2024gromov}.
For compactly supported marginals in $\mathbb R^{d_x}$ and
$\mathbb R^{d_y}$, the empirical estimator based on independent samples satisfies
\[
\mathbb E\!\left[
\left|\mathsf{GW}_{2,2}^{\,2}(\mu_n,\nu_n)
-\mathsf{GW}_{2,2}^{\,2}(\mu,\nu)\right|
\right]
\le
C n^{- \frac{2}{(d_x\wedge d_y)\vee 4}}
(\log n)^{\mathbf 1_{\{d_x\wedge d_y=4\}}}.
\]
Matching worst-case lower bounds are also established for the empirical estimator, up to the logarithmic factor. The dependence on the smaller dimension $d_x\wedge d_y$ is related to the lower-complexity adaptation principle developed in \cite{groppe2024lower}.
%whose analysis covers both unregularized and entropic GW. 
This near-optimal rate of convergence for $(2,2)$-GW is extended to unbounded supports assuming only finite polynomial moments in \cite{kato2025convergence}, along with matching minimax lower bounds up to logarithmic factors and deviation inequalities. 
For compactly supported marginal distributions, recent preprints extend empirical upper bounds beyond $(2,2)$-GW with the same rate. For instances, Euclidean even-order $(2r,2k)$-GW is treated in \cite{paliy2026even}, $(2,1)$-GW problem is studied in \cite{leung2026sample}, and limit laws for unregularized GW objectives are established in for the Euclidean $(2,2)$ formulation in discrete and suitably regular semi-discrete settings \cite{RiouxGoldfeldKatoLimitLaws2024} and for $(p,q)$ formulation on arbitrary finite metric measure spaces \cite{RiouxMarksPasseggeriGoldfeld2026}. In addition, \cite{RiouxGoldfeldKatoLimitLaws2024} also studies stability of the unregularized $(2,2)$-GW functional under perturbations of the marginal distributions in discrete and semi-discrete settings.

A complementary literature develops statistically and computationally tractable alternatives to GW through modified or relaxed functionals.
These include the statistical and computational development of second lower bound (distribution of distances) \cite{weitkamp2024distances} as a lower bound on GW \cite{Mem11}, and sliced inner-product GW discrepancy \cite{gong2026sliced,pal2025wasserstein_alignment}.

\paragraph{Regularized GW problems.}
For entropically regularized Euclidean $(2,2)$-GW, a variational representation through a finite-dimensional family of entropic OT problems yields parametric rate for the empirical regularized value under light tail assumptions
\cite{zhang_ziv2024gromov}, namely,
\[
\mathbb E\,\Bigl[
\bigl|
S_{2,2,\varepsilon}^{\text{ent}}(\mu_n,\nu_n)
-S_{2,2,\varepsilon}^{\text{ent}}(\mu,\nu)
\bigr|
\Bigr]
\leq
C n^{-\frac{1}{2}}.
\]
Lower-complexity adaptation principle also applies under compact Euclidean supports \cite{groppe2024lower}.
Distributional inference for entropic (2,2)-GW under moment constraints is developed in \cite{RiouxGoldfeldKatoLimitLaws2024}. 
Beyond squared-Euclidean geometry, parametric $n^{-1/2}$ empirical convergence rates and limit distributions for entropic GW on arbitrary finite spaces with general dissimilarity kernels are established in \cite{RiouxMarksPasseggeriGoldfeld2026}. Neural estimators of entropic GW are developed and analyzed through non-asymptotic bounds \cite{wang2026neural}.

Stability of the entropic $(2,2)$-GW functional under perturbations of the marginal distributions has been established under light tail assumptions \cite{RiouxGoldfeldKatoLimitLaws2024}.
Separately, smoothness of the matrix variational objective in \eqref{eq:intro-dual} with respect to $A$ is established in \cite{rioux_kengo2024entropicGW}, with a lower bound on $\varepsilon$ ensuring strict convexity and thus uniqueness of the entropic GW optimizer.

\paragraph{Related advances in regularized OT problems.}
Parametric estimation for entropic OT has been established in \cite{genevay2019sample,mena2019statistical}. 
Subsequent work develops related limit theory; see \cite{Alberto_Sanz_Gonzalez_del_Barrio_CLT_survey25,LiChen2025_EMOT} and references therein. 
Recent advances show that general sparse divergence does not necessarily preclude parametric statistical behavior, contrary to the conventional wisdom
\cite{BayraktarEcksteinZhang_stability}. Central limit theorems under broad sparse regularizations are established in \cite{AlbertoGonzalez-SanzEcksteinStephanNutzMarcel2025Lpweaklimit, gonzalez2025quadratic} and uniform central limit theorems and bootstrap consistency are established in \cite{pengtao_alberto_xiaohui_2026centrallimittheoremsbootstrap}. Finite-sample bias and variance bounds are obtained in \cite{alberto.gonzalez-marcel.nutz_austin.stromme2026}.

On the stability side, quantitative Wasserstein bounds for entropic
optimal couplings under marginal perturbations are obtained in
\cite{eckstein2022quantitative}, later extended to a broad class of divergence penalties in
\cite{BayraktarEcksteinZhang_stability}. Quantitative stability for quadratically regularized OT is further
established in \cite{gonzalez2026stability}.

\subsection{Notations}
For a Polish space $(\cZ, d_\cZ)$, we use 
$\cP(\cZ)$ to denote the Borel probability
measures on $\cZ$. For $p<\infty$, $\cP_p(\cZ)$ denotes the Borel probability measures with finite $p$-th moment. 
For $p=\infty$, we define $\cP_\infty(\cZ)$ as the measures with bounded support.
For each $p\in[1,\infty]$, equip
$\cX\times\cY$ with the product metric
\[
d_{\cX\times\cY,p}((x,y),(x',y'))
:=
\begin{cases}
\bigl(
d_{\cX}(x,x')^p + d_{\cY}(y,y')^p\bigr)^{1/p}, & p<\infty,\\[0.4em]
d_{\cX}(x,x') \vee d_{\cY}(y,y') & p=\infty,
\end{cases}
\]
where $a\vee b := \max\{a,b\}$.
Throughout, the \(p\)-Wasserstein distance on $W_p$ on $\cX\times\cY$ is induced by the particular product metric $d_{\cX\times\cY,p}$, with the value $+\infty$ allowed.
Specifically, the distance between two tuples of Borel probability measures $(\mu,\nu)$ and $(\tilde \mu,\tilde \nu)$
will often be quantified by
\[
\Delta:=W_p((\mu,\nu),(\tilde \mu,\tilde \nu))
:=
\begin{cases}
\left(
W_p^p(\mu,\tilde \mu) + W_p^p(\nu,\tilde \nu)\right)^{1/p}, & p<\infty,\\[0.6em]
W_\infty(\mu,\tilde \mu) \vee W_\infty(\nu,\tilde \nu), & p=\infty.
\end{cases}
\]

For a finite signed measure $\sigma$ and a measurable function $r\ge0$ satisfying $\int_{\cX\times\cY}r\,d|\sigma|<\infty$,
write $r\sigma(A):=\int_A r\,d\sigma$ and use the full variation norm 
$\|r\sigma\|_{\mathrm{TV}}
:=|r\sigma|(\cX\times\cY)
=\int_{\cX\times\cY}r\,d|\sigma|$.
We write $B(x,r)$ for the open Euclidean ball of radius $r$ centered at $x$, and $\overline B(x,r)$ for its closure. 
For $f:\cX\to\bR$ and $g:\cY\to\bR$, set $(f\oplus g)(x,y)=f(x)+g(y)$.
For matrices $A,B$, $\langle A,B\rangle_F=\tr(A^\top B),\,
 \|A\|_F^2=\langle A,A\rangle_F.$ Given an integer $i$, we write $ [i] := \{1,\dots,i\}$.

\subsection{Organization of the paper}
The remainder of the paper is structured as follows. \Cref{sec: preliminaries} provides the main background results needed for our analysis. 
The quantitative stability of optimizer of regularized $(r,s)$-Gromov-Wasserstein distance is presented in \Cref{sec:quantitative_stability_regularized_GW}. 
The sample complexity of estimating regularized $(2,2)$-GW is given in \Cref{sec:sample_complexity_regularized _Gromov_Wasserstein_distance}. Appendix \ref{appendix_omitted_proofs} gathres omitted proof details.

\section{Background and preliminaries}\label{sec: preliminaries}

\subsection{Background on divergence regularized optimal transport}

The regularized optimal transport problem (ROT) is defined by
\begin{align}\label{eq:intro_def_ROT}
\mathrm{OT}_{\varepsilon}^c(\mu,\nu)
:=
\inf_{\pi\in\Pi(\mu,\nu)}
\biggl\{ 
\int c\,d\pi + \varepsilon \,\mathsf D_{\varphi}(\pi,\mu\otimes\nu)
\biggr\}
\end{align}
for
$c:\cX\times \cY\to\mathbb{R}$ continuous cost function.
When $\mu$ and $\nu$ have compact supports $\cX$ and $\cY$, respectively, ROT attains strong duality, i.e., the value of the primal problem \eqref{eq:intro_def_ROT} equals that of the dual problem 
\[
\sup_{\substack{f\in\cC(\cX) \\ g\in\cC(\cY)}}
\biggl\{
\Phi_{\mu,\nu}^{c}(f,g)
: =
\int
\Bigl(
f(x)+g(y)
-
\varepsilon\,
\psi\,
\Bigl(
\frac{f(x)+g(y)-c(x,y)}{\varepsilon}
\Bigr)
\Bigr)
\,d(\mu\otimes \nu)(x,y)
\biggr\}
\]
where $\psi(s) := \sup_{t\geq 0}(st-\varphi(t))$ is the convex conjugate of $\varphi$.
The dual problem admits a (non-unique) optimizer $(f_*,g_*)\in L^\infty(\mu)\times L^\infty(\nu)$.
One can find a version of the optimizer $(f_*,g_*)$ of the dual problem satisfying the first-order condition
\begin{equation}\label{eqn:schrodinger_system}
\begin{cases}
\displaystyle
\int
\psi'\,
\biggl(
\frac{f_*(x)+g_*(y)-c(x,y)}{\varepsilon}
\biggr)\,d\nu(y)
=1,
& \text{for all } x\in\cX,\\[1.2ex]
\displaystyle
\int
\psi'\,
\biggl(
\frac{f_*(x)+g_*(y)-c(x,y)}{\varepsilon}
\biggr)\,d\mu(x)
=1,
& \text{for all } y\in\cY.
\end{cases}
\end{equation}
For such a version $(f_*,g_*)$, the corresponding primal optimal coupling $\pi_\varepsilon^*$ for \eqref{eq:intro_def_ROT} is unique and of the form 
$$
\frac{d \pi_\varepsilon^*}{d (\mu\otimes\nu)}(x,y)
=\psi'\Bigl(\frac{f_*(x) + g_*(y) - c(x,y)}{\varepsilon}\Bigr).
$$
Moreover, the uniform bound
$\|f_*\oplus g_*\|_\infty
\le
5\|c\|_\infty+\varepsilon
<\infty$
holds on $\cX\times\cY$.
If $c\in \cC(\cX\times\cY)$ admits a uniform modulus of continuity $\omega$, then $f_*$ and $g_*$ both admit $\omega$.
For further properties of ROT, see \cite{pengtao_alberto_xiaohui_2026centrallimittheoremsbootstrap,alberto.gonzalez-marcel.nutz_austin.stromme2026}.
\begin{proposition}\cite[Corollary 3.2]{alberto.gonzalez-marcel.nutz_austin.stromme2026}\label{proposition:ROT_bias_parametric_bound}
Under \Cref{assumption:regularizaion} and compact support assumption, 
\begin{equation}
\Bigl|
\bE\OT^c(\mu_n,\nu_n)-\OT^c(\mu,\nu)
\Bigr|
\leq
2\sqrt{2}\,
(5\|c\|_\infty + \varepsilon)
\frac{
\psi'(\frac{6\|c\|_\infty + \varepsilon}{\varepsilon})
}{\sqrt{n}}.
\label{eq:imported-bias}
\end{equation}
\end{proposition}

The following lemma provides the first connection between regularized GW and regularized OT.
\begin{lemma}
\label{lemma:GW-linearization}
Suppose $\mu\in\cP_{rs}(\cX)$ and $\nu\in\cP_{rs}(\cY)$.
If $\pi^*$ is a finite minimizer of
\[
\pi\longmapsto
\cE_\Gamma(\pi)+\varepsilon \,\mathsf{D}_\varphi(\pi,\mu\otimes\nu)
\quad\text{over }\Pi(\mu,\nu),
\]
then we have for 
$c_*(x,y):=2\int\Gamma((x,y),(x',y'))\,d\pi^*(x',y')$,
\begin{equation}
\label{eq:GW-linearized-OT}
\pi^*\in
\operatorname*{argmin}_{\theta\in\Pi(\mu,\nu)}
\left\{
\int c_*\,d\theta
+
\varepsilon \,\mathsf{D}_\varphi(\theta,\mu\otimes\nu)
\right\}.
\end{equation}
\end{lemma}
%\textcolor{pink}{This might be used as an iterative algorithm to compute GW optimizer.}

\subsection{Preliminaries}
\begin{proposition}[Finiteness and primal attainment]
\label{proposition:GW-attainment}
Suppose \Cref{assumption:regularizaion} holds.
Suppose that $\mu\in\cP_{rs}(\cX)$ and $\nu\in\cP_{rs}(\cY)$
for $r,s\in[1,\infty)$.
For every $\varepsilon>0$, the unregularized and regularized
problems have nonempty, weakly compact sets of minimizers, and
\begin{equation}\label{eq:gw-finite-value-bounds}
0
\le \mathsf{GW}_{r,s}(\mu,\nu)^r
\le S_{r,s,\varepsilon}(\mu,\nu)
\le \cE_\Gamma(\mu\otimes\nu)
<\infty.
\end{equation}
Every regularized minimizer $\pi^*$ satisfies
\[
\mathsf{D}_\varphi(\pi^*,\mu\otimes\nu)
\le \varepsilon^{-1}\cE_\Gamma(\mu\otimes\nu).
\]
\end{proposition}

\begin{lemma}[Regularity of the conjugate]
\label{lemma:conjugate-regularity}
Suppose \Cref{assumption:regularizaion} holds.
Then $\psi$ is finite and belongs to $\cC^1(\bR)$, with
$\psi'\ge0$, $\psi'(1)=1$, and $\psi'(s)\to\infty$ as $s\to\infty$.
Moreover, $\psi$ is strictly convex on $[1-\delta,\infty)$
for some $\delta>0$.
For every compact interval $I$, there exists $L_I<\infty$ such that
\begin{equation}
\label{eq:holder-psi}
|\psi'(s)-\psi'(t)|\le L_I|s-t|^\frac{1}{1+\beta},
\qquad s,t\in I.
\end{equation}
For every $\theta\in(0,1)$, there also exists
$\kappa_{I,\theta}>0$ such that
\begin{equation}
\label{eq:active-curvature}
B_\psi(s,t):=\psi(s)-\psi(t)-\psi'(t)(s-t)
\ge\kappa_{I,\theta}|s-t|^2
\quad\text{if }s,t\in I,\ \psi'(t)\ge\theta.
\end{equation}
\end{lemma}

\begin{assumption}[Moment conditions]
\label{assumption:Agamma}
Let $r,s\in[1,\infty)$ and $p\in[1,\infty]$.
A pair of probability measures $(\mu,\nu)\in\cP(\cX)\times\cP(\cY)$ satisfies this assumption for $(r,s,p)$ if the applicable condition below holds:
\begin{itemize}
\item
For $1<p<\infty$, 
$\mu\in\cP_{a_p}(\cX), \nu\in\cP_{a_p}(\cY)$  with 
$
a_p:=
p\vee\frac{(rs-1)p}{p-1}
.
$

\item
For $p=1$ and $rs=1$,  $\mu\in\cP_1(\cX),
\nu\in\cP_1(\cY)$.
\item 
For $p=1$ and $rs>1$, both $\mu$ and $\nu$
have bounded support.
\item
For  $p=\infty$, $\mu\in\cP_{rs}(\cX),
\nu\in\cP_{rs}(\cY)$ .
\end{itemize}
\end{assumption}

\begin{proposition}[Weakened Lipschitz continuity]\label{proposition:Agamma}
Suppose that both $(\mu,\nu)$ and $(\tilde\mu,\tilde\nu)$ satisfy \Cref{assumption:Agamma} for the same $r,s\in[1,\infty)$ and $p\in[1,\infty]$. Then there exists $\Lambda>0$ such that
\begin{equation}\label{eq:Agamma}
\left|\cE_{\Gamma_{r,s}}(\pi)-\cE_{\Gamma_{r,s}}(\widetilde\pi)\right|
\le\Lambda W_p(\pi,\widetilde\pi)
\end{equation}
for all $\pi\in\Pi(\mu,\nu)$ and
$\widetilde\pi\in\Pi(\tilde\mu,\tilde\nu)$.
\end{proposition}
We define, for $a,b\ge0$,
\[
b_\varphi(a\mid b)
:=
\lim_{t\downarrow0}
\frac{(1-t)\varphi(b)+t\varphi(a)
-\varphi((1-t)b+ta)}{t}
\in[0,\infty].
\]
The limit exists by convexity and agrees with
$\varphi(a)-\varphi(b)-\varphi'(b)(a-b)$ when $b>0$.
Moreover, $b_\varphi(0\mid0)=0$, and, for $a>0$,
$b_\varphi(a\mid0)
=\varphi(a)-\varphi(0)-a\varphi'_+(0)$,
with value $+\infty$ when $\varphi'_+(0)=-\infty$.
For probability measures $Q,R\ll \mu\otimes\nu$, set
\begin{equation*}
B^{\mu,\nu}_\varphi(Q\mid R)
:=
\int_{\cX\times\cY}
b_\varphi\!\left(
\frac{dQ}{d (\mu\otimes\nu) }\,\middle|\,\frac{dR}{d (\mu\otimes\nu)}
\right)d (\mu\otimes\nu).
\end{equation*}
As ensured by \Cref{proposition:GW-attainment}, minimizers exist and take $\pi^*\in\operatorname*{argmin}_{\pi\in\Pi(\mu,\nu)}\cF_\varepsilon^{\mu,\nu}(\pi)$.
\begin{assumption}[Curvature domination]\label{assumption:Cgamma}
We assume
$
\varepsilon>\kappa
$
for $\kappa\ge0$ such that for every $\theta\in\Pi(\mu,\nu)$ with 
$
\mu\in\mathcal P_{rs}(\mathcal X),
\nu\in\mathcal P_{rs}(\mathcal Y)
$
and
$\cF_\varepsilon^{\mu,\nu}(\theta)<\infty$,
\begin{equation}\label{eq:Cgamma}
\cE_\Gamma(\theta-\pi^*)
\ge-\kappa B^{\mu,\nu}_\varphi(\theta\mid\pi^*).
\end{equation}

\end{assumption}

\begin{remark}
The assumption on the regularization parameter $\varepsilon$ to be sufficiently large to guarantee the uniqueness of the optimizer is common in the literature related to GW distance; see, e.g., \cite[Theorem 6]{rioux_kengo2024entropicGW}. The following verification of \Cref{assumption:Cgamma} is in the same spirit.
\end{remark}

\begin{remark}[Verification of \Cref{assumption:Cgamma}]
\label{remark:Cgamma-verification}
We provide sufficient conditions for \Cref{assumption:Cgamma} under entropy regularization and Tsallis regularization.
Let $r,s\in[1,\infty)$, and let
$\pi^*$ be a finite minimizer of $\cF_\varepsilon^{\mu,\nu}$.
Fix $(x_0,y_0)\in\cX\times\cY$.
Assume
$\mu\in\cP_{2rs}(\cX)$ and $\nu\in\cP_{2rs}(\cY)$, and set
\[
w(x,y):=
1+d_{\cX}(x,x_0)^{rs}+d_{\cY}(y,y_0)^{rs},
\qquad
S:=\sup_{\eta\in\Pi(\mu,\nu)}\int w^2\,d\eta.
\]
Then
\begin{align}\label{inequality_gamma_bound_assumption_verificaion}
%\int w^2\,d \mu\otimes\nu\le 
S\le
3\left(
1+\int d_{\cX}(x,x_0)^{2rs}\,d\mu(x)
+\int d_{\cY}(y,y_0)^{2rs}\,d\nu(y)
\right)<\infty.
\end{align}
(i) Entropy regularization. It suffices to require
$
\varepsilon
>\kappa
=2^{rs-1}S.
$

\noindent
(ii) Tsallis regularization with $1<\alpha\le2$.
$
\varepsilon
>\kappa
=
\frac{2^{rs-1}}{\alpha-1}
\left(\int w^2\,d(\mu\otimes\nu)\right)^{\alpha-1}S^{2-\alpha}
$
would be sufficient.

\noindent
(iii) Tsallis regularization with $\alpha>2$ on compact supports. 
It suffices to require $\varepsilon \geq 5M$
with $M:=
\max\left\{
\operatorname{diam}(\operatorname{supp}\mu)^{rs},
\operatorname{diam}(\operatorname{supp}\nu)^{rs}
\right\}$. Under this setting,
$
\varepsilon>\kappa
:=
\frac{\alpha M}{2(\alpha-1)}
\left(1-\frac{4M}{\varepsilon}\right)^{-(\alpha-2)/(\alpha-1)}.
$
\end{remark}

\section{Quantitative stability of regularized Gromov-Wasserstein distance}\label{sec:quantitative_stability_regularized_GW}

We use the notion of a
shadow for the quantitative stability results.
This notion was first introduced by \cite{eckstein2022quantitative, EcksteinNutz_shadow}. 
The shadow of a coupling $\pi\in\Pi(\mu,\nu)$ is a particular
$W_p$-projection of $\pi$ onto $\Pi(\tilde\mu,\tilde\nu)$.
In addition to this metric property, the shadow has divergence
relative to the new product measure no greater than that of
$\pi$ relative to the original product measure.
These two properties make shadows useful for studying the
stability of regularized optimal transport.

\begin{definition}[Shadow]
\label{definition:shadow}
Let $p\in[1,+\infty]$ and $\mu,\tilde\mu\in\mathcal P_p(\mathcal X)$ and
$\nu,\tilde\nu\in\mathcal P_p(\mathcal Y)$.
Choose optimal transport plans
\[
\theta_{\mathcal X}
=
\mu\otimes K_{\mathcal X}
\in\Pi(\mu,\tilde\mu),
\qquad
\theta_{\mathcal Y}
=
\nu\otimes K_{\mathcal Y}
\in\Pi(\nu,\tilde\nu)
\]
for $W_p(\mu,\tilde\mu)$ and $W_p(\nu,\tilde\nu)$, respectively,
where $K_{\mathcal X}$ and $K_{\mathcal Y}$ are stochastic
kernels obtained by disintegration with respect to the first
marginals. Define the product kernel
\[
K\bigl((x,y),\cdot\bigr)
:=
K_{\mathcal X}(x,\cdot)
\otimes
K_{\mathcal Y}(y,\cdot).
\]
For $\pi\in\Pi(\mu,\nu)$, its shadow
$\tilde\pi\in\Pi(\tilde\mu,\tilde\nu)$, associated with the chosen
plans $\theta_{\mathcal X}$ and $\theta_{\mathcal Y}$, is the
second marginal of
\[
\pi\otimes K
\in
\mathcal P\bigl(
(\mathcal X\times\mathcal Y)
\times
(\mathcal X\times\mathcal Y)
\bigr).
\]
Equivalently,
\[
\tilde\pi(d\tilde x,d\tilde y)
=
\int_{\mathcal X\times\mathcal Y}
K_{\mathcal X}(x,d\tilde x)
K_{\mathcal Y}(y,d\tilde y)
\,\pi(dx,dy).
\]
\end{definition}

\begin{proposition}\cite[Lemma~3.2]{eckstein2022quantitative}\label{proposition_shadow_property}
For $\pi\in\Pi(\mu,\nu)$, its shadow $\tilde\pi\in\Pi(\tilde\mu,\tilde\nu)$ defined above satisfies 
\[
W_p(\pi,\tilde\pi)
=
W_p\bigl((\mu,\nu),(\widetilde\mu,\widetilde\nu)\bigr),
\qquad
\mathsf{D}_\varphi\bigl(
\tilde\pi,\tilde\mu\otimes\tilde\nu
\bigr)
\le
\mathsf{D}_\varphi\bigl(
\pi,\mu\otimes\nu
\bigr).
\]
\end{proposition}
The shadow need not be unique, since it may depend on the chosen optimal transport plans.

\begin{theorem}[Continuity of value]\label{thm:continuity_regualarized_GW_value}
Suppose $(\mu,\nu)$ and $(\tilde\mu,\tilde\nu)$ both satisfy \Cref{assumption:Agamma}. Then
\begin{align*}
|S_{r,s,\varepsilon}(\mu,\nu) 
- 
S_{r,s,\varepsilon}(\tilde\mu,\tilde\nu)| 
\leq
\Lambda W_p( (\mu,\nu),(\tilde\mu,\tilde\nu) )
\end{align*}
\end{theorem}

\begin{proof}
Let $\pi^*$ be an optimizer of $S_{r,s,\varepsilon}(\mu,\nu)$ and $\widetilde\pi\in\Pi(\tilde\mu,\tilde\nu)$ be its shadow. \Cref{proposition_shadow_property} gives that
\begin{align*}
S_{r,s,\varepsilon}(\mu,\nu)
&=
\cE_\Gamma (\pi^*) + \varepsilon \,\mathsf{D}_\varphi(\pi^*,\mu\otimes\nu)\\
&\geq 
\cE_\Gamma (\tilde\pi) - \Lambda W_p((\mu,\nu),(\tilde\mu,\tilde\nu)) + \varepsilon \,\mathsf{D}_\varphi(\widetilde\pi,\tilde\mu\otimes\tilde\nu)\\
&\geq 
S_{r,s,\varepsilon}(\tilde\mu,\tilde\nu) - \Lambda W_p((\mu,\nu),(\tilde\mu,\tilde\nu)). 
\end{align*}  
The claim follows by symmetry.
\end{proof}

\begin{lemma}\label{lemma:strong-convexity-tv}
Suppose \Cref{assumption:regularizaion} and  \Cref{assumption:Cgamma} hold. If
$\pi\in\Pi(\mu,\nu)$ and $\pi^*$ is the optimizer of $S_{r,s,\varepsilon}$
and $h:\cX\times\cY\to[0,\infty)$ is measurable with
$\int_{\cX\times\cY}h^2\,d(\mu\otimes\nu+\pi^*+\pi)<\infty$,
then
\begin{align}\label{eq:objective-tv}
\|h(\pi&-\pi^*)\|_{\mathrm{TV}}^{\beta+2}\nonumber\\
&\le
\frac{ (\beta+1)(\beta+2)3^{\beta/2}((\lambda_0+\lambda_1)\vee\lambda_2)}{\varepsilon-\kappa}
\|h\|_{L^2(\mu\otimes\nu+\pi^*+\pi)}^{\beta+2}
\bigl(\cF_\varepsilon^{\mu,\nu}(\pi)
-
\cF_\varepsilon^{\mu,\nu}(\pi^*)\bigr).
\end{align}
\end{lemma}

\begin{proof}
Without loss of generality, we assume $\cF_\varepsilon^{\mu,\nu}(\pi) < \infty$. Otherwise  the inequality to show is trivial.
We first show that
\begin{equation}\label{eq:coercivity}
\cF_\varepsilon^{\mu,\nu}(\pi)-\cF_\varepsilon^{\mu,\nu}(\pi^*)
\ge
(\varepsilon-\kappa)B^{\mu,\nu}_\varphi(\pi\mid\pi^*).
\end{equation}
For $t\in[0,1]$, let $\pi_t:=(1-t)\pi^*+t\pi$.
Then
\begin{equation}\label{eq:energy-expansion}
\cE_\Gamma(\pi_t)
=(1-t)\cE_\Gamma(\pi^*)+t\cE_\Gamma(\pi)
-t(1-t)\cE_\Gamma(\pi-\pi^*).
\end{equation}
Write $f_*:=d\pi^*/d(\mu\otimes\nu)$, $f:=d\pi/d(\mu\otimes\nu)$, and
$f_t:=(1-t)f_*+tf$. The convexity gap
\[
J_t
:=(1-t)\,\mathsf{D}_\varphi(\pi^*,\mu\otimes\nu)+t\,\mathsf{D}_\varphi(\pi,\mu\otimes\nu)
-\mathsf{D}_\varphi(\pi_t,\mu\otimes\nu)
\]
is finite and nonnegative. Combining its defining identity with
\eqref{eq:energy-expansion} and using optimality gives
\begin{equation}\label{eq:near-optimal}
\cF_\varepsilon^{\mu,\nu}(\pi)-\cF_\varepsilon^{\mu,\nu}(\pi^*)
\ge
(1-t)\cE_\Gamma(\pi-\pi^*)+\varepsilon\frac{J_t}{t},
\qquad 0<t<1.
\end{equation}
The integrand defining $J_t$ is nonnegative. Hence Fatou's lemma,
the definition of $B^{\mu,\nu}_\varphi$, and \eqref{eq:near-optimal} imply
\begin{align*}
B^{\mu,\nu}_\varphi(\pi\mid\pi^*)
&=
\int_{\cX\times\cY}\lim_{t\downarrow0}
\frac{(1-t)\varphi(f_*)+t\varphi(f)-\varphi(f_t)}{t}\,d\mu\otimes\nu\\
&\le\liminf_{t\downarrow0}\frac{J_t}{t}
\le
\frac{
\cF_\varepsilon^{\mu,\nu}(\pi)-\cF_\varepsilon^{\mu,\nu}(\pi^*)-\cE_\Gamma(\pi-\pi^*)
}{\varepsilon}
<\infty.
\end{align*}
Thus the Bregman divergence is finite, and rearranging gives
\[
\cF_\varepsilon^{\mu,\nu}(\pi)-\cF_\varepsilon^{\mu,\nu}(\pi^*)
\ge
\cE_\Gamma(\pi-\pi^*)+\varepsilon B^{\mu,\nu}_\varphi(\pi\mid\pi^*)
\ge
(\varepsilon-\kappa)B^{\mu,\nu}_\varphi(\pi\mid\pi^*),
\]
where the last step uses \eqref{eq:Cgamma}.
This proves \eqref{eq:coercivity}.
Next, \Cref{assumption:regularizaion} gives, for $u>0$,
\[
\varphi''(u)
\ge
\frac{u^{\beta}}
{\lambda_0+\lambda_1u^{\beta}
+\lambda_2u^{\beta+1}}
\ge
\frac{u^{\beta}}
{((\lambda_0+\lambda_1)\vee\lambda_2)(1+u)^{\beta+1}}.
\]
Take $a,b>0$, it holds that
\[
\int_0^1(1-s)((1-s)b+sa)^\beta\,ds
\ge
\frac{(a\vee b)^\beta}{(\beta+1)(\beta+2)}.
\]
Indeed, if \(a\ge b\), then $(1-s)b+sa\ge sa$,
and hence
\[
\int_0^1(1-s)((1-s)b+sa)^\beta\,ds
\ge
a^\beta\int_0^1(1-s)s^\beta\,ds
=
\frac{a^\beta}{(\beta+1)(\beta+2)}.
\]
If \(b\ge a\), then $(1-s)b+sa\ge (1-s)b$,
so
\[
\int_0^1(1-s)((1-s)b+sa)^\beta\,ds
\ge
b^\beta\int_0^1(1-s)^{\beta+1}\,ds
=
\frac{b^\beta}{\beta+2}
\ge
\frac{b^\beta}{(\beta+1)(\beta+2)}.
\]
Taylor's formula therefore yields
\begin{align}\label{eq:pointwise-bregman-lower}
b_\varphi(a\mid b)
&=(a-b)^2\int_0^1(1-s)
\varphi''((1-s)b+sa)\,ds
\\
&\ge
\frac{(a-b)^2(a \vee b)^{\beta}}
{(\beta+1)(\beta+2)((\lambda_0+\lambda_1)\vee\lambda_2)
(1+a \vee b)^{\beta+1}}\nonumber\\
&\ge
\frac{|a-b|^{\beta+2}}
{(\beta+1)(\beta+2)((\lambda_0+\lambda_1)\vee\lambda_2)
(1+a+b)^{\beta+1}}\nonumber.
\end{align}
The same inequality holds for \(a,b\ge0\): apply
\eqref{eq:pointwise-bregman-lower} to \(a+\delta,b+\delta\) and let
\(\delta\downarrow0\). By the boundary definition of \(b_\varphi\), $b_\varphi(a+\delta\mid b+\delta)\to b_\varphi(a\mid b)$,
with convergence to \(+\infty\) allowed when
\(\varphi'_+(0)=-\infty\).
Applying \eqref{eq:pointwise-bregman-lower} to $f,f_*$ and using
H\"older's inequality with
$k:=\beta+2$ and $k':=k/(k-1)\le2$ yields
\begin{align}
&\|h(\pi-\pi^*)\|_{\mathrm{TV}}^k
=\left(\int_{\cX\times\cY}h|f-f_*|\,d\mu\otimes\nu\right)^k
\notag\\
&\le
\left(\int_{\cX\times\cY}
\frac{|f-f_*|^k}{(1+f+f_*)^{k-1}}\,d\mu\otimes\nu\right)
\left(\int_{\cX\times\cY}h^{k'}(1+f+f_*)\,d\mu\otimes\nu\right)^{k-1}
\notag\\
&\le
(\beta+1)(\beta+2)((\lambda_0+\lambda_1)\vee\lambda_2)
B^{\mu,\nu}_\varphi(\pi\mid\pi^*)
\left(\int_{\cX\times\cY}h^{k'}\,d(\mu\otimes\nu+\pi+\pi^*)\right)^{k-1}
\notag\\
&\le
(\beta+1)(\beta+2)3^{\beta/2}((\lambda_0+\lambda_1)\vee\lambda_2)
\left(\int_{\cX\times\cY}h^2\,d(\mu\otimes\nu+\pi+\pi^*)\right)^{k/2}
B^{\mu,\nu}_\varphi(\pi\mid\pi^*).
\label{eq:bregman-tv}
\end{align}
The last inequality uses that $\mu\otimes\nu+\pi+\pi^*$ has total mass $3$:
\[
\left(\int_{\cX\times\cY}h^{k'}\,d(\mu\otimes\nu+\pi+\pi^*)\right)^{k-1}
\le
3^{\beta/2}
\left(\int_{\cX\times\cY}h^2\,d(\mu\otimes\nu+\pi+\pi^*)\right)^{k/2}.
\]
Combining \eqref{eq:bregman-tv} with
\eqref{eq:coercivity} proves \eqref{eq:objective-tv}.
\end{proof}

\begin{corollary}[Uniqueness under curvature domination]
\label{corollary:GW-uniqueness}
Under \Cref{assumption:regularizaion} and \Cref{assumption:Cgamma}, the minimizer of $\cF^{\mu,\nu}_\varepsilon(\cdot)$ over $\Pi(\mu,\nu)$ is unique.
\end{corollary}

\begin{proof}
Apply \Cref{lemma:strong-convexity-tv} with
$\pi=\theta$ and weight $h\equiv1$. Since $\int_{\cX\times\cY}1\,d(\mu\otimes\nu+\theta+\pi^*)=3$
and
$\cF_\varepsilon^{\mu,\nu}(\pi^*)=
S_{r,s,\varepsilon}(\mu,\nu)$,
the lemma gives for every
$\theta\in\Pi(\mu,\nu)$,
\[
\|\theta-\pi^*\|_{\mathrm{TV}}^{\beta+2}
\le
\frac{(\beta+1)(\beta+2)3^{1+\beta}((\lambda_0+\lambda_1)\vee\lambda_2)}
{\varepsilon-\kappa}
\bigl(
\cF_\varepsilon^{\mu,\nu}(\theta)
-
S_{r,s,\varepsilon}(\mu,\nu)
\bigr).
\]
The proof completes.
\end{proof}

\begin{theorem}[Divergence-regularized GW stability]
\label{theorem:DGW-stability}
Let $r,s\in[1,\infty)$, $p\in[1,\infty]$, and $1\le q<\infty$ with $q\le p$.
Suppose that $\varphi$ satisfies
\Cref{assumption:regularizaion}, and let
$\pi^*$ and $\widetilde\pi^*$ be finite minimizers of
$S_{r,s,\varepsilon}(\mu,\nu)$ and
$S_{r,s,\varepsilon}(\tilde\mu,\tilde\nu)$, respectively.
Assume that  \Cref{assumption:Agamma} holds for the same $(r,s,p)$ and let $\Lambda>0$ be a constant satisfying \eqref{eq:Agamma} as provided by \Cref{proposition:Agamma}.
Assume further that \Cref{assumption:Cgamma} holdss at $\pi^*$ with $\kappa <\varepsilon$.
Assume $\mu\in\cP_{2q}(\cX)$ and $\nu\in\cP_{2q}(\cY)$
and 
let $\Delta = W_p((\mu,\nu),(\tilde\mu,\tilde\nu))$.
Then
\begin{align}\label{eq:main-bound}
W_q(\pi^*,\widetilde\pi^*)
\le
C_1\Delta
+
C_2 \Bigl(
\frac{\Lambda\Delta}{\varepsilon-\kappa}\Bigr)^{1/[q(\beta+2)]},
\end{align}
with $C_1,C_2$ depending on $\beta,\lambda_0,\lambda_1,\lambda_2,\varepsilon$ and $2q$-th moments of $\mu$ and $\nu$.
\begin{comment}
\begin{align}\label{eq:main-bound}
W_q(\pi^*,\widetilde\pi^*)
\le
2^{1/q-1/p}\Delta
+
2^{1-1/q}B_{\mu,\nu,q}^{1/(2q)}
\left(
\frac{2 (\beta+1)(\beta+2)3^{\beta/2} ((\lambda_0+\lambda_1)\vee\lambda_2)\Lambda\Delta}
{\varepsilon-\kappa}
\right)^{1/[q(\beta+2)]},
\end{align}
\end{comment}
\end{theorem}

\begin{proof}
Let $\pi\in\Pi(\mu,\nu)$ be the $p$-shadow of $\widetilde\pi^*$. \Cref{proposition_shadow_property} and \Cref{proposition:Agamma} indicate that 
\begin{align*}
\cF_\varepsilon^{\mu,\nu} (\pi) -  \cF_\varepsilon^{{\tilde\mu,\tilde\nu}} (\widetilde\pi^*)
=
\cE_\Gamma(\pi) - \cE_\Gamma(\widetilde\pi^*) + \varepsilon ( \mathsf{D}_\varphi (\pi, \mu\otimes\nu) - \mathsf{D}_\varphi (\widetilde\pi^*, \tilde\mu\otimes\tilde\nu) ) \\
\leq
\cE_\Gamma(\pi) - \cE_\Gamma(\widetilde\pi^*)
\leq
\Lambda W_p(\pi, \widetilde\pi^*) = \Lambda \Delta.
\end{align*}
Since $\pi^*$ minimizes $\cF_\varepsilon^{\mu,\nu}(\cdot)$ over $\Pi(\mu,\nu)$,
\begin{align}\label{eq:first-shadow-objective}
    \cF_\varepsilon^{\mu,\nu}(\pi^*) 
    \leq 
    \cF_\varepsilon^{\mu,\nu}(\pi) 
    \leq
    \cF_\varepsilon^{{\tilde\mu,\tilde\nu}}(\widetilde\pi^*) + \Lambda \Delta.
\end{align}
Similarly,
 \begin{align}\label{eq:reverse-shadow-objective}
     \cF_\varepsilon^{{\tilde\mu,\tilde\nu}}(\widetilde\pi^*) 
     \leq
     \cF_\varepsilon^{\mu,\nu}(\pi^*)
     +
     \Lambda \Delta.
 \end{align}
Combining \eqref{eq:first-shadow-objective} and
\eqref{eq:reverse-shadow-objective} yields the value-gap estimate 
\begin{equation}\label{eq:shadow-objective-gap}
0\leq
\cF_\varepsilon^{\mu,\nu}(\pi)-\cF_\varepsilon^{\mu,\nu}(\pi^*)
=
\bigl(
\cF_\varepsilon^{\mu,\nu}(\pi) - \cF_\varepsilon^{\tilde\mu,\tilde\nu}(\tilde\pi^*)
\bigr)
+
\bigl(
\cF_\varepsilon^{\tilde\mu,\tilde\nu}(\tilde\pi^*)
-
\cF_\varepsilon^{\mu,\nu}(\pi^*)
\bigr)
\le2\Lambda\Delta.
\end{equation}
Since  $\mu\in\cP_{2q}(\cX)$ and $\nu\in\cP_{2q}(\cY)$, Cauchy-Schwarz inequality gives that %because $(\sum_i a_i)^2\le m\sum_i a_i^2$.
\begin{align*}
B_{\mu,\nu,q}
&:=
\sup_{\eta_1,\eta_2\in\Pi(\mu,\nu)}
\int_{\cX\times\cY}(d_{\cX}(x,x_{0})^q + d_{\cY}(y,y_{0})^q)^2\,d(\mu\otimes\nu+\eta_1+\eta_2)(x,y)\\
&\le
6
\int_{\cX}d_{\cX}(x,x_{0})^{2q}\,d\mu(x)
+
6
\int_{\cY}d_{\cY}(y,y_{0})^{2q}\,d\nu(y)
<\infty,
\end{align*}
for any fixed $(x_{0},y_{0})\in\cX\times\cY$. Apply \eqref{eq:objective-tv} with $h=w_q:=
d_{\cX}(x,x_{0})^q + d_{\cY}(y,y_{0})^q$. Since
$\pi^*,\pi\in\Pi(\mu,\nu)$, \eqref{eq:shadow-objective-gap} gives
\begin{equation}\label{eq:weighted-tv-shadow}
\|w_q(\pi^*-\pi)\|_{\mathrm{TV}}^{\beta+2}
\le
\frac{2(\beta+1)(\beta+2)3^{\beta/2}((\lambda_0+\lambda_1)\vee\lambda_2)
B_{\mu,\nu,q}^{(\beta+2)/2}\Lambda\Delta}
{\varepsilon-\kappa}.
\end{equation}
The weighted variation estimate
\cite[Proposition~7.10]{villani2021topics} states that
\begin{equation}\label{eq:Wq-tv}
W_q(\pi^*,\pi)^q
\le
2^{q-1}\|w_q(\pi^*-\pi)\|_{\mathrm{TV}}.
\end{equation}
Combining \eqref{eq:weighted-tv-shadow} and \eqref{eq:Wq-tv}, we obtain 
\begin{equation}\label{eq:Wq-first}
W_q(\pi^*,\pi)
\le
2^{1-1/q}B_{\mu,\nu,q}^{1/(2q)}
\left(
\frac{2(\beta+1)(\beta+2)3^{\beta/2}((\lambda_0+\lambda_1)\vee\lambda_2)\Lambda\Delta}
{\varepsilon-\kappa}
\right)^{1/[q(\beta+2)]}.
\end{equation}
For every $\gamma\in\Pi(\pi,\widetilde\pi^*)$, using H\"{o}lder inequality twice gives
\begin{align*}
\left(\int d_{\cX\times\cY,q}(z,\tilde z )^q\,d\gamma(z,\tilde z)\right)^{1/q}
&\le
2^{1/q-1/p}
\left(\int d_{\cX\times\cY,p}(z,\tilde z)^q\,d\gamma(z,\tilde z)\right)^{1/q}\\
&\le
2^{1/q-1/p}
\left(\int d_{\cX\times\cY,p}(z,\tilde z)^p\,d\gamma(z,\tilde z)\right)^{1/p}
\end{align*}
when $p<\infty$, with the essential-supremum interpretation when
$p=\infty$. Taking infima over $\gamma$ and using
\Cref{proposition_shadow_property} yields
\begin{equation}\label{eq:Wq-second}
W_q(\pi,\widetilde\pi^*)
\le
2^{1/q-1/p}W_p(\pi,\widetilde\pi^*)
=2^{1/q-1/p}\Delta.
\end{equation}
The triangle inequality, together with \eqref{eq:Wq-first} and
\eqref{eq:Wq-second}, proves \eqref{eq:main-bound}.
\end{proof}

\section{Sample complexity of regularized Gromov-Wasserstein distance}\label{sec:sample_complexity_regularized _Gromov_Wasserstein_distance}

Throughout this section, we consider regularized Euclidean $(2,2)$-GW functional $S_{2,2,\varepsilon}(\mu,\nu)$
with any fixed regularization parameter $\varepsilon>0$
and assume
probability measures
$\mu\in\mathcal P(\bR^{d_x})$ and
$\nu\in\mathcal P(\bR^{d_y})$ with compact supports $ \mathcal X\subset\overline B(0,R_x)$, $ \mathcal Y\subset\overline B(0,R_y)$, respectively, where $R_x,R_y\ge1$. Observe that $S_{2,2,\varepsilon}$ is invariant to isometric operations on the marginal spaces, such as translation and orthonormal rotation. 
Thus, we assume without loss of generality that $\mu$ and $\nu$ are centered, i.e.,
$\int x\,d\mu(x)=\int y\,d\nu(y)=0$.

We next derive a dual formulation for the regularized GW distance. 
The representation \eqref{eq:gw-variational-duality} reduces
divergence-regularized GW to a family of regularized OT problems. In particular, an outer minimizer $A^*$ and its dual-optimal
potentials determine a regularized GW-optimal coupling through \eqref{eq:optimal-matrix-moment}.
This duality connects the regularized GW distance to the recently well-understood theory of ROT \cite{AlbertoGonzalez-SanzEcksteinStephanNutzMarcel2025Lpweaklimit,pengtao_alberto_xiaohui_2026centrallimittheoremsbootstrap,alberto.gonzalez-marcel.nutz_austin.stromme2026}, and serves as a key ingredient in our sample complexity analysis of the empirical regularized GW distance.

For compactly supported probability measures $\lambda\in\cP(\bR^{d_x})$ and $\eta\in\cP(\bR^{d_y})$, write
$m_2(\lambda):=\int\|x\|^2d\lambda(x)$,
$m_2(\eta):=\int\|y\|^2\,d\eta(y)$,
$M_{\lambda,\eta}:=\sqrt{m_2(\lambda)m_2(\eta)}$, 
$\mathcal D_M:=\{A\in\bR^{d_x\times d_y}:\|A\|_F\le M/2\}$ for $M\ge0$ and put $\mathcal D:=\mathcal D_{R_xR_y} =\{A\in\bR^{d_x\times d_y}:\|A \|_F\le R_xR_y/2\}$.
Also we define 
\begin{align}\label{eqn:def_GW1}
\mathsf{GW}^1(\lambda,\eta)
:=\iint\|x-x'\|^4\,d\lambda(x)d\lambda(x')
      +\iint\|y-y'\|^4\,d\eta(y)d\eta(y')-4m_2(\lambda)m_2(\eta).
\end{align}
For $A\in\bR^{d_x\times d_y}$, set $ c_A(x,y):=-4\|x\|^2\|y\|^2-32x^\top Ay$.

\begin{theorem}[Regularized GW duality]
\label{thm:gw-duality}
Let $\mu\in\cP(\bR^{d_x})$ and $\nu\in \cP(\bR^{d_y})$ be compactly supported. For every $M\ge M_{\mu,\nu}$,
\begin{align}\label{eq:gw-variational-duality}
&\quad S_{2,2,\varepsilon}(\mu,\nu) -\mathsf{GW}^1(\mu,\nu)\\
=&
\min_{A\in\bR^{d_x\times d_y}}
\left\{32\|A\|_F^2+\OT^{c_A}(\mu,\nu)\right\}
=
\min_{A\in\mathcal D_M}
\left\{32\|A\|_F^2+\OT^{c_A}(\mu,\nu)\right\}.\nonumber
\end{align} Equivalently,
\begin{equation}
\begin{split}
S_{2,2,\varepsilon}(\mu,\nu)-\mathsf{GW}^1(\mu,\nu)
=
&\min_{A\in\mathcal D_M}
\max_{\substack{f\in \cC(\cX)\\g\in \cC(\cY)}}
\biggl\{
32\|A\|_F^2+\int f\, d\mu+\int g\,d\nu \\
&-\varepsilon\int
\psi\!\left(
\frac{f(x)+g(y)-c_A(x,y)}{\varepsilon}
\right)\,d(\mu\otimes\nu)(x,y)
\biggr\}.
\end{split}
\label{eq:gw-potential-duality}
\end{equation}
For each \(A\in\mathbb{R}^{d_x\times d_y}\), the inner ROT problem admits a unique optimizer \(\pi_A\).
The GW minimizers and the outer minimizers in either matrix domain,
correspond bijectively through
\begin{equation}
\pi^*=\pi_{A^*},
\qquad
A^*=\frac12\int xy^\top\,d\pi^*(x,y).
\label{eq:optimal-matrix-moment}
\end{equation}
\end{theorem}

\begin{proof}[Proof of \Cref{thm:gw-duality}]
The optimums related to regularized GW are attained because of the lower semicontinuity of $\pi$ and $A$.
For each fixed \(A\), the ROT problem has a unique primal optimizer \(\pi_A\), and its dual maximum is attained \cite[Proposition~2.3]{alberto.gonzalez-marcel.nutz_austin.stromme2026}.
The minimum in \eqref{eq:gw-variational-duality} is attained due to continuity of the objective with respect to $A$ on the compact set $\cD$.
For $\pi\in\Pi(\mu,\nu)$, write $a_\mu:=\int x\,d\mu(x), b_\nu:=\int y\,d\nu(y), B_\pi:=\int xy^\top\,d\pi(x,y)$.
Then $\langle A,B_\pi\rangle_F
 =\int\langle A,xy^\top\rangle_F\,d\pi
 =\int x^\top Ay\,d\pi$.
Expanding squared distances, noting that $\iint\langle x,x'\rangle\langle y,y'\rangle\,d\pi\,d\pi = \|B_\pi\|^2_F$ and then integrating gives
\begin{equation}
\mathcal{E}_{\Gamma}(\pi)
=\mathsf{GW}^{1}(\mu,\nu)
-4\int\|x\|^2\|y\|^2\,d\pi
-8\|B_\pi\|_F^2
+8\int\bigl(
a_\mu^\top x\|y\|^2
+b_\nu^\top y\|x\|^2
\bigr)\,d\pi.
\label{eq:distortionexpansion}
\end{equation}
\begin{comment}
\begin{align}
\mathcal E_\Gamma(\pi)
%&=\iint\|x-x'\|^4\,d\mu(x)d\mu(x')+\iint\|y-y'\|^4\,d\nu(y)d\nu(y')-2\iint\|x-x'\|^2\|y-y'\|^2\,d\pi(x,y)d\pi(x',y') \nonumber\\
%& =\mathsf{GW}^1(\mu,\nu)-4\int\|x\|^2\|y\|^2\,d\pi -8\iint\langle x,x'\rangle\langle y,y'\rangle\,d\pi\,d\pi+8a_\mu^\top\int x\|y\|^2\,d\pi+8b_\nu^\top\int y\|x\|^2\,d\pi \nonumber\\
&= 
\mathsf{GW}^1(\mu,\nu)-4\int\|x\|^2\|y\|^2\,d\pi
        -8\|B_\pi\|_F^2
    +8a_\mu^\top\int x\|y\|^2\,d\pi
     +8b_\nu^\top\int y\|x\|^2\,d\pi.
\label{eq:distortionexpansion}
\end{align}
\end{comment}
The last term vanishes when the marginals are centered, as assumed.
By the triangle and Cauchy--Schwarz inequalities,
\[
 \|B_\pi\|_F
 \le\int\|xy^\top\|_F\,d\pi
 =\int\|x\|\|y\|\,d\pi
 \le\sqrt{m_2(\mu)m_2(\nu)}=M_{\mu,\nu}.
\]
Consequently, completing the square yields
\begin{align}\label{eq:gw-completed-square}
32\|A\|_F^2+\int c_A\,d\pi
&+\varepsilon\,\mathsf{D}_\varphi(\pi,\mu\otimes\nu)
+ \mathsf{GW}^1(\mu,\nu) - S_{2,2,\varepsilon}(\mu,\nu)\nonumber\\
=&
\cF_\varepsilon^{\mu,\nu}(\pi)
-S_{2,2,\varepsilon}(\mu,\nu)
+32\Bigl\|A-\frac12B_\pi\Bigr\|_F^2.
\end{align}
Since $B_\pi/2\in\mathcal D_M$, taking infima over $\pi$ and $A$
in either order proves \eqref{eq:gw-variational-duality}.

If $\pi^*$ is regularized GW-optimal, then
\eqref{eq:gw-completed-square} shows that
$\pi^*$ is the optimizer of the regularized OT problem with cost $c_{A^*}$ for $A^*=B_{\pi^*}/2$.
Hence $A^*$ minimizes the outer objective in
either matrix domain, and $\pi^*$ solves the fixed-$A^*$ OT
problem. Thus both outer infima are attained and  uniqueness of regularized OT optimizer gives $\pi^*=\pi_{A^*}$.
Conversely, for any outer minimizer $A^*$, the same identity gives
\[
0=
\cF_\varepsilon^{\mu,\nu}(\pi_{A^*})
-S_{2,2,\varepsilon}(\mu,\nu)
+
32\Bigl\|A^*-\frac12B_{\pi_{A^*}}\Bigr\|_F^2.
\]
Both terms on the right are nonnegative. Hence $\pi_{A^*}$ is regularized GW-optimal and $A^*=B_{\pi_{A^*}}/2$, proving
\eqref{eq:optimal-matrix-moment} and the asserted bijection.

For each fixed $A$, bounded-cost transport duality
\cite[Proposition~2.3]{alberto.gonzalez-marcel.nutz_austin.stromme2026} gives dual formulation of ROT problem 
and then substituting this attained fixed-$A$ dual representation into \eqref{eq:gw-variational-duality} proves \eqref{eq:gw-potential-duality}.
\end{proof}

To analyze the empirical regularized GW value through variational representation \eqref{eq:gw-variational-duality}, we need uniform estimates over $A\in\mathcal D$.
We begin with bounds for the cost family $\{c_A:A\in\mathcal D\}$, which will also provide uniform bounds and regularity estimates for the associated optimal dual potentials.
\begin{lemma}
\label{lem:gw-cost-bounds}
Under the assumption of 
Theorem~\ref{thm:gw-duality}. We have
\begin{align}
\sup_{A\in D}\|c_A\|_\infty\le 20R_x^2R_y^2,
\quad
\|c_A-c_{A'}\|_\infty\le32R_xR_y\|A-A'\|_F, \,\, A,A'\in \mathcal D.\nonumber\\
 |c_A(x,y)-c_A(x',y)|\le24R_xR_y^2\|x-x'\|,
 \quad
 |c_A(x,y)-c_A(x,y')|\le24R_x^2R_y\|y-y'\|.
 \label{eq:gw-spatial-cost-bounds}
\end{align}
\end{lemma}

The preceding estimates, together with the properties of ROT dual potentials recalled in \Cref{sec: preliminaries}, allow us to collect the population optimal dual integrands into a uniformly bounded class. We now define this class. Write $B=100R_x^2R_y^2+\varepsilon$ and set
\[
 I=
 \left[-\frac{2B+20R_x^2R_y^2}{\varepsilon},
\frac{2B+20R_x^2R_y^2}{\varepsilon}\right].
\]
There exists a H\"older constant $L\ge1$ such that
\[
 |\psi'(s)-\psi'(t)|
 \le L|s-t|^{\frac{1}{1+\beta}},
 \qquad s,t\in I,
\]
as provided by \eqref{eq:holder-psi}.
In view of \Cref{thm:gw-duality}, we consider ROT problem with cost $c_A$ for each $A\in\mathcal D$. 
We normalize the associated optimal dual potentials in \eqref{eqn:schrodinger_system} by requiring $\max_{\mathcal X}f+\min_{\mathcal X}f=0$ and define
$$
\cS = \{ (A,f,g); A\in \mathcal D, (f,g)\in\mathcal C(\mathcal X)\times\mathcal C(\mathcal Y)\,\,\text{solve}\,\,\eqref{eqn:schrodinger_system}\,\text{with cost}\,c_A,
\max_{\mathcal X}f+\min_{\mathcal X}f=0
\}.$$
Define
\begin{equation*}
 H_{A,f,g}
 :=
 f\oplus g
 -\varepsilon\,\psi\,
 \Bigl(
\frac{f\oplus g-c_A}{\varepsilon}
\Bigr),
\qquad
\cH:=\{H_{A,f,g}:(A,f,g)\in\cS\}.
\end{equation*}
Thus, $\cH$ is uniformly bounded
since every $(A,f,g)\in\cS$ satisfies
$\|f\oplus g\|_\infty\le B$,
$\|f\|_\infty\le B$ and $\|g\|_\infty\le B$ due to the normalization $\max_{\mathcal X}f+\min_{\mathcal X}f=0$,
and the potentials inherit the respective coordinatewise Lipschitz constants in Lemma 
\ref{lem:gw-cost-bounds}.
By Lemma~\ref{lemma:conjugate-regularity}, there is a constant
$\kappa=\kappa_{I,1/2}>0$ such that
\eqref{eq:active-curvature} holds on $I$ with cutoff $1/2$. In what follows, $L$ and $\kappa$ denote these fixed constants.

To use $\mathcal H$ in the empirical-process argument, we next quantify its complexity in the uniform norm. The following lemma derives a covering bound from a stability estimate in which the cost matrix and finitely many cell averages of one dual potential control the associated dual integrand. This yields a finite-dimensional H\"older parametrization of $\mathcal H$, providing a key ingredient in the proof of \Cref{thm:parametric_sample_complexity}.
Related stability results for quadratically regularized OT are
studied in \cite{gonzalez2026stability} under stronger assumptions
and by different methods.

\begin{lemma}[Covering number bound for $\cH$]
\label{lem:covering}
Then there exist an integer $J$ satisfying
\begin{equation}
 1\le J
 \le
 \left(
 1+\frac{192R_x^2R_y^2(4L)^{(1+\beta)}}{\varepsilon}
 \right)^{d_x}
 \label{eq:covering-J-bound}
\end{equation}
and a constant $C\ge1$ depending on $R_x,d_x,R_y,d_y,\varphi,\varepsilon$ such that 
\begin{equation}
 \log N(\eta,\cH,\|\cdot\|_\infty)
 \le
2(1+\beta)(d_xd_y+J)
\log(1+ C/\eta),
 \qquad \eta>0.
 \label{eq:entropy-bound}
\end{equation}
\end{lemma}
Define, for $\lambda\in\mathcal P(\mathcal X)$ and $\eta\in\mathcal P(\mathcal Y)$, 
$V_\varepsilon(\lambda,\eta):=\min_{A\in\mathcal D}
\left\{32\|A\|_F^2+\OT^{c_A}(\lambda,\eta)\right\}$.
\begin{lemma}
\label{lem:empirical-reduction}
Under the assumption of  Theorem~\ref{thm:gw-duality}, 
we have, 
\begin{equation}
\begin{aligned}
|S_{2,2,\varepsilon}(\mu_n,\nu_n)-S_{2,2,\varepsilon}(\mu,\nu)|
\le
|\mathsf{GW}^1(\mu_n,\nu_n)-&\mathsf{GW}^1(\mu,\nu)|
+
|V_\varepsilon(\mu_n,\nu_n)-V_\varepsilon(\mu,\nu)|\\
&\quad+
8R_xR_y^2\|\mu_nx\|
+
8R_x^2R_y\|\nu_ny\|.
\end{aligned}
\label{eq:empirical-gw-reduction}
\end{equation}
\end{lemma}

We now combine these ingredients to obtain the sample complexity
bound. 
\begin{theorem}[Parametric rate of estimation of the regularized GW functional]
\label{thm:parametric_sample_complexity}
Suppose Assumption~\ref{assumption:regularizaion} holds
and $\mu,\nu$ have compact supports. Then there exists a finite constant
$C\geq0$ depending on 
$\varphi,R_x,d_x,R_y,d_y,\varepsilon$, such that
\begin{align}
 \bE\bigl|S_{2,2,\varepsilon}(\mu_n,\nu_n)-S_{2,2,\varepsilon}(\mu,\nu)\bigr|^2
\le Cn^{-1}.\label{eq:main-l2}
\end{align}

\end{theorem}

\begin{proof}[Proof of Theorem~\ref{thm:parametric_sample_complexity}]
We divide the proof into three steps.

\noindent
\textit{Step 1: mean absolute error bound.}
We first prove
\begin{align}\label{eq:main-l1}
\bE\bigl|S_{2,2,\varepsilon}(\mu_n,\nu_n)-S_{2,2,\varepsilon}(\mu,\nu)\bigr|\,\,\,
\le C n^{-1/2},
\end{align}
for some constant $C\geq0$ depending on  $\varphi,R_x,d_x,R_y,d_y,\varepsilon$. 
By \Cref{lem:empirical-reduction}, it suffices to bound $|\mathsf{GW}^1(\mu_n,\nu_n)
-
\mathsf{GW}^1(\mu,\nu)|, |
V_\varepsilon(\mu_n,\nu_n)
-
V_\varepsilon(\mu,\nu)|, \|\mu_nx\|$ and $
\|\nu_ny\|$.
Define $ S_n:=\sup_{h\in\cH}|((\mu_n\otimes\nu_n) - (\mu\otimes\nu) )h|$. 
Hoeffding's inequality for two-sample U-statistics \cite{hoeffding1963probability} implies that $Z_h:=((\mu_n\otimes\nu_n) - (\mu\otimes\nu) )h$ has sub-Gaussian increments with respect to \(n^{-1/2}\|\cdot\|_\infty\).
Then \Cref{lem:covering} combined with chaining \cite[Theorem 5.22]{Wainwright_2019} gives
\begin{equation}
 \bE S_n\le C n^{-1/2}.
 \label{eq:Sn-bound}
\end{equation}
Abbreviate
\[
 \mathrm{OT}_A:=\OT^{c_A}(\mu,\nu),\quad
 \widehat{\mathrm{OT}}_A:=\OT^{c_A}(\mu_n,\nu_n),\quad
 V:=V_\varepsilon (\mu,\nu),\quad V_n:=V_\varepsilon(\mu_n,\nu_n).
\]
Theorem~\ref{thm:gw-duality} guarantees a population minimizer $A^*\in \mathcal D$. Strong duality for ROT gives $\mathrm{OT}_{A^*}=(\mu\otimes\nu) H_{A^*}$.
Define 
\begin{equation*}
 Z_{{A^*},n}:=((\mu_n\otimes\nu_n) -(\mu\otimes\nu) )H_{A^*},\qquad
 R_{{A^*},n}:=\widehat{\mathrm{OT}}_{A^*}-(\mu_n\otimes\nu_n)  H_{A^*}\ge0.
\end{equation*}
The inequality is empirical duality.
%\bE Z_{{A^*},n}=0.
Thus \Cref{proposition:ROT_bias_parametric_bound} gives
\begin{equation}
 \bE R_{{A^*},n}
 =\bE\widehat{\mathrm{OT}}_{A^*}-\mathrm{OT}_{A^*}
 \le 
 C n^{-1/2}.
 \label{eq:remainder-mean}
\end{equation}
Evaluation at $A^*$ gives
\[
 V_n-V\le {\widehat{\mathsf {OT}}}_{A^*}
 -
 \mathsf {OT}_{A^*}
 =Z_{A^*,n}+R_{A^*,n}\le S_n+R_{A^*,n}.
\]
For the opposite direction, take any $A\in\mathcal D$ with associated integrand $H_A\in\mathcal H$.  Empirical duality and the definition of $S_n$ imply
\[
 \widehat{\mathrm{OT}}_A
 \ge (\mu_n\otimes\nu_n)H_A
 \ge (\mu\otimes\nu)H_A-S_n
 = \mathrm{OT}_A-S_n.
\]
Taking infimum over $A\in\mathcal D$ yields
\[
 V_n
 =\inf_{A\in\mathcal D}
   \bigl\{32\|A\|_F^2+\widehat{\mathrm{OT}}_A\bigr\}
 \ge
 \inf_{A\in\mathcal D}
   \bigl\{32\|A\|_F^2+\mathrm{OT}_A\bigr\}-S_n
 =V-S_n.
\]
As a consequence, we know that
\begin{equation}
 |V_n-V|\le R_{A^*,n}+S_n.
 \label{eq:one-sided-transfer}
\end{equation}
By 
\Cref{eq:Sn-bound,eq:remainder-mean,eq:one-sided-transfer},
\begin{equation}
 \bE|V_n-V|
 \le
 C
 n^{-1/2}.
 \label{eq:outer-rate}
\end{equation}
Cauchy-Schwarz inequality gives that
\begin{align}\label{eqn:moment_bound}
\mathbb E|m_2(\mu_n)-m_2(\mu)|
\le \sqrt{\frac{\operatorname{Var}(\|X_1\|^2)}{n}}
\le \frac{R_x^2}{\sqrt n},
\qquad
\mathbb E|m_2(\nu_n)-m_2(\nu)|
\le \frac{R_y^2}{\sqrt n}.
\end{align}
For $h(x,x'):=\|x-x'\|^4 \leq 16R_x^4$, set $Q:=\iint h\,d(\mu\otimes\mu)$, $Q_n:=\frac1{n^2}\sum_{1\leq i,j\leq n} h(X_i,X_j)$.
Since $h(x,x)=0$, we have $\mathbb E Q_n=(1-1/n)Q$.
Independence gives
\[
\begin{aligned}
\operatorname{Var}(Q_n)
&=
\frac1{n^4}\sum_{1\leq i,j,k,l\leq n}
\operatorname{Cov}\bigl(h(X_i,X_j),h(X_k,X_l)\bigr)\\
&\le
\frac{(16R_x^4)^2}{n^4}\sum_{1\leq i,j,k,l\leq n}
\left(
\mathbf 1_{\{i=k\}}+\mathbf 1_{\{i=l\}}
+\mathbf 1_{\{j=k\}}+\mathbf 1_{\{j=l\}}
\right)
=
\frac{4(16R_x^4)^2}{n}.\\
\mathbb E|Q_n-Q|
&\le \mathbb E|Q_n- \bE Q_n| + |\bE Q_n - Q|
\le\sqrt{\operatorname{Var}(Q_n)}+\frac Qn
\le\frac{48R_x^4}{\sqrt n}.
\end{aligned}
\]
The same bounds hold for $\nu_n$, with $R_x$ replaced by $R_y$.
At the same time, note that
\begin{align*}
|m_2(\mu_n)m_2(\nu_n)-m_2(\mu)m_2(\nu)|
\le 
R_y^2|m_2(\mu_n)-m_2(\mu)|
+
R_x^2|m_2(\nu_n)-m_2(\nu)|
\end{align*}
and it follows that
\begin{align}\label{eqn:GW1_bound}
 \bE|\mathsf{GW}^1(\mu_n,\nu_n)-\mathsf{GW}^1(\mu,\nu)|\le 
 C n^{-1/2}.
\end{align}
At the same time, Jensen's inequality implies that
\begin{align}\label{eqn:empirical_moment_bound}
\mathbb E\|\mu_nx\|
\le
\biggl(
\mathbb E\Bigl\|\frac1n\sum_{i=1}^nX_i\Bigr\|^2
\biggr)^{1/2}
=
\sqrt{\frac{m_2(\mu)}n}
\le\frac{R_x}{\sqrt n},
\qquad
\mathbb E\|\nu_ny\|
\le
\sqrt{\frac{m_2(\nu)}n}
\le\frac{R_y}{\sqrt n}.
\end{align}
Combining \Cref{eq:empirical-gw-reduction,eqn:moment_bound,eqn:GW1_bound,eq:outer-rate,eqn:empirical_moment_bound} gives
\begin{align*}
 \bE|S_{2,2,\varepsilon}(\mu_n,\nu_n)-S_{2,2,\varepsilon}(\mu,\nu)|
 \le&\bE|\mathsf{GW}^1(\mu_n, \nu_n)
 -\mathsf{GW}^1(\mu,\nu)|
       +\bE|V_n-V|\\
 &\quad+8R_xR_y^2\bE\|\mu_nx\|
+8R_x^2R_y\bE\|\nu_ny\|
\le C n^{-1/2}.
\end{align*}
\textit{Step 2: reformulation of regularized GW under empirical input.}
We prove \Cref{eq:labeled-grid} in this step. Let
$K_\Gamma
:=
\|\Gamma_{2,2}\|_{\infty,(\mathcal X\times\mathcal Y)^2}$ and   
\[
\mathcal W_n
:=
\biggl\{
w\in[0,\infty)^{n\times n}:
\sum_{j=1}^n w_{ij}=\frac1n\ \text{for all }i\in[n],\quad
\sum_{i=1}^n w_{ij}=\frac1n\ \text{for all }j\in[n]
\biggr\}.
\]
Given
$\mathbf z=(x_1,\ldots,x_n,y_1,\ldots,y_n)
\in\mathcal X^n\times\mathcal Y^n$,
define
\[
F_{\mathbf z}(w)
:=
\sum_{1\leq i,j,k,l \leq n}
\Gamma_{\mathbf z}^{ij,kl}w_{ij}w_{kl}
+
\frac{\varepsilon}{n^2}
\sum_{1\leq i,j\leq n}\varphi(n^2w_{ij}), \quad
\Gamma_{\mathbf z}^{ij,kl}
:=
\Gamma((x_i,y_j),(x_k,y_l)).
\]
We claim
\begin{equation}
G_n(\mathbf z)
:=
S_{2,2,\varepsilon}\biggl(
\frac1n\sum_{i=1}^n\delta_{x_i},
\frac1n\sum_{j=1}^n\delta_{y_j}
\biggr)
=
\inf_{w\in\mathcal W_n}F_{\mathbf z}(w).
\label{eq:labeled-grid}
\end{equation}
Write $\mu_{\mathbf z}:=\frac1n\sum_{i=1}^n\delta_{x_i}, \nu_{\mathbf z}:=\frac1n\sum_{j=1}^n\delta_{y_j}$.
Let $\mathcal A:=\{x_1,\ldots,x_n\}$ and
$\mathcal B:=\{y_1,\ldots,y_n\}$ be the sets of distinct
observed locations. For $a\in\mathcal A$ and $b\in\mathcal B$,
define
\[
I_a:=\{i\in[n]:x_i=a\},\qquad
J_b:=\{j\in[n]:y_j=b\},\qquad
N_a:=|I_a|,\qquad M_b:=|J_b|.
\]
Every $\pi\in\Pi(\mu_{\mathbf z},\nu_{\mathbf z})$ is supported on $\mathcal A\times\mathcal B$, and we write $\pi_{ab}:=\pi(\{(a,b)\})$. 
%Since the empirical product measure assigns mass $N_aM_b/n^2>0$ to $(a,b)$, such a coupling is absolutely continuous with respect to $\mu_{\mathbf z}\otimes\nu_{\mathbf z}$. 
Then we have
\begin{align}\label{eqn:discrete_divergence}
\mathsf D_\varphi(\pi,\mu_{\mathbf z}\otimes\nu_{\mathbf z})
=
\sum_{a\in\mathcal A}\sum_{b\in\mathcal B}
\frac{N_aM_b}{n^2}
\varphi\!\left(\frac{n^2\pi_{ab}}{N_aM_b}\right).
\end{align}
For $w\in\mathcal W_n$, define
$
\pi^w:=\sum_{1\leq i,j\leq n} w_{ij}\delta_{(x_i,y_j)}.
$
The row and column constraints imply that
$\pi^w\in\Pi(\mu_{\mathbf z},\nu_{\mathbf z})$, and
\begin{align}\label{eqn:discrete_distortion}
\iint\Gamma(u,v)\,d\pi^w(u)\,d\pi^w(v)
=
\sum_{1\leq i,j,k,l\leq n}
\Gamma_{\mathbf z}^{ij,kl}w_{ij}w_{kl}.
\end{align}
Moreover, with
$\pi^w_{ab}:=\sum_{i\in I_a,j\in J_b}w_{ij}$,
Jensen's inequality gives
\[
\frac1{n^2}
\sum_{\substack{i\in I_a\\j\in J_b}}
\varphi(n^2w_{ij})
\ge
\frac{N_aM_b}{n^2}
\varphi\!\left(\frac{n^2\pi^w_{ab}}{N_aM_b}\right).
\]
Summing over $a\in\mathcal A,b\in\mathcal B$ shows that
\[
\mathsf D_\varphi(\pi^w,\mu_{\mathbf z}\otimes\nu_{\mathbf z})
\le
\frac1{n^2}\sum_{i,j=1}^n\varphi(n^2w_{ij}).
\]
Consequently, $G_n(\mathbf z)\le F_{\mathbf z}(w)$ for every
$w\in\mathcal W_n$, and hence
\begin{align}\label{eqn:discrete_formula_proof_first}
G_n(\mathbf z)\le
\inf_{w\in\mathcal W_n}F_{\mathbf z}(w).
\end{align}
Conversely, given
$\pi\in\Pi(\mu_{\mathbf z},\nu_{\mathbf z})$, define $w^\pi_{ij}:=\frac{\pi_{ab}}{N_aM_b}$
for $i\in I_a, j\in J_b$.
For $i\in I_a$, 
\[
\sum_{\ell=1}^n w^\pi_{i\ell}
=
\sum_{b'\in\cB}\sum_{\ell\in J_{b'}}w^\pi_{i\ell}
=
\sum_{b'\in\cB}\sum_{\ell\in J_{b'}}\frac{\pi_{ab'}}{N_aM_{b'}}
=
\frac1{N_a}\sum_{b'\in\mathcal B}\pi_{ab'}
=
\frac{\mu_{\mathbf z}(\{a\})}{N_a}
=
\frac1n.
\]
Analogously, $\sum_{k=1}^n w^\pi_{kj}
=
\frac1{M_b}\sum_{a'\in\mathcal A}\pi_{a'b}
=
\frac1n$ for $j\in J_b$.
Thus $w^\pi\in\mathcal W_n$, and its induced coupling is
$\pi^{w^\pi}=\pi$. In addition,
\[
\frac1{n^2}\sum_{i,j=1}^n\varphi(n^2w^\pi_{ij})
=
\sum_{a\in\mathcal A}\sum_{b\in\mathcal B}
\frac{N_aM_b}{n^2}
\varphi\!\left(\frac{n^2\pi_{ab}}{N_aM_b}\right)
=
\mathsf D_\varphi(\pi,\mu_{\mathbf z}\otimes\nu_{\mathbf z}).
\]
Therefore,
\[
F_{\mathbf z}(w^\pi)
=
\iint\Gamma(u,v)\,d\pi(u)\,d\pi(v)
+
\varepsilon
D_\varphi(\pi,\mu_{\mathbf z}\otimes\nu_{\mathbf z}).
\]
Taking the infimum over $\pi$ gives
\begin{align}\label{eqn:discrete_formula_proof_second}
\inf_{w\in\mathcal W_n}F_{\mathbf z}(w)\le G_n(\mathbf z).
\end{align}
Combining \Cref{eqn:discrete_formula_proof_first,eqn:discrete_formula_proof_second} proves \eqref{eq:labeled-grid}. 

\noindent
\textit{Step 3: Efron--Stein inequality.} 
Suppose now that $\mathbf z'$ differs from $\mathbf z$ only
in the coordinate $x_r$. As a consequence, %For fixed $w\in\mathcal W_n$, the coefficient $\Gamma_{\mathbf z}^{ij,kl}$ can change only when $i=r$ or $k=r$. 
\begin{align*}
|F_{\mathbf z}(w)-F_{\mathbf z'}(w)|
\le
K_\Gamma
\sum_{\substack{1\le i,j,k,l\le n\\
                 i=r\ \text{or}\ k=r}}
w_{ij}w_{kl}
%=K_\Gamma \biggl(2\Bigl(\sum_{j=1}^nw_{rj}\Bigr)\Bigl(\sum_{1\leq k,l\leq n} w_{kl}\Bigr)-\Bigl(\sum_{j=1}^n w_{rj}\Bigr)^2\biggr)
%=K_\Gamma\left(\frac2n-\frac1{n^2}\right)
\le
\frac{2K_\Gamma}{n}.
\end{align*}
Taking infima gives
\[
|G_n(\mathbf z)-G_n(\mathbf z')|
\le
\sup_{w\in\mathcal W_n}
|F_{\mathbf z}(w)-F_{\mathbf z'}(w)|
\le
\frac{2K_\Gamma}{n}.
\]
The same bound holds analogously when a single coordinate $y_r$ is
replaced.

For $\mathbf Z=(X_1,\ldots,X_n,Y_1,\ldots,Y_n)$, denote $\widehat G:=G_n(\mathbf Z)=S_{2,2,\varepsilon}(\mu_n,\nu_n)$.
For $r\in [2n]$, let $\mathbf Z^{(r)}$ denote the sample
vector obtained by replacing its $r$-th coordinate with an
independent copy, and set
$\widehat G^{(r)}:=G_n(\mathbf Z^{(r)})$.
Applying the Efron--Stein inequality \cite{efronstein1981jackknife,lugosiboucheron2003concentration} yields
%to the $2n$ independent observations 
\[
\operatorname{Var}(\widehat G)
\le
\frac12\sum_{r=1}^{2n}
\mathbb E\left[
(\widehat G-\widehat G^{(r)})^2
\right]
%\le\frac12(2n)\left(\frac{2K_\Gamma}{n}\right)^2
\le
\frac{4K_\Gamma^2}{n}.
\]
Finally, the bias--variance identity and \eqref{eq:main-l1} give
\[
\mathbb E\left|\widehat G-S_{2,2,\varepsilon}(\mu,\nu)\right|^2
=
\operatorname{Var}(\widehat G)
+
\left(\mathbb E\widehat G-S_{2,2,\varepsilon}(\mu,\nu)\right)^2
\le
\frac{4K_\Gamma^2+C^2}{n}.
\]
\end{proof}

\bibliography{GW_stability}

\newpage
\appendix

\begin{center}
    \textbf{\huge Appendix}
\end{center}

\section{Omitted proofs}\label{appendix_omitted_proofs}

\begin{proof}[Proof of \Cref{proposition:GW-attainment}]
Fix $x_0\in\cX$ and $y_0\in\cY$. Since
$|a^s-b^s|^r\le a^{rs}+b^{rs}$ for $a,b\ge0$,
the marginal constraints and the triangle inequality give,
for every $\pi,\theta\in\Pi(\mu,\nu)$,
\begin{align*}
\iint\Gamma\,d\pi\,d\theta
&\le
\iint_{\cX^2}d_{\cX}(x,x')^{rs}\,d\mu(x)d\mu(x')
+
\iint_{\cY^2}d_{\cY}(y,y')^{rs}\,d\nu(y)d\nu(y')
\notag\\
&\le
2^{rs}\left(
\int_{\cX}d_{\cX}(x,x_0)^{rs}\,d\mu(x)
+
\int_{\cY}d_{\cY}(y,y_0)^{rs}\,d\nu(y)
\right)<\infty.
\end{align*}
Since
$\mathsf{D}_\varphi(\mu\otimes\nu,\mu\otimes\nu)=\varphi(1)=0$,
\eqref{eq:gw-finite-value-bounds} is proved.

The set $\Pi(\mu,\nu)$ is uniformly tight and weakly closed,
thus weakly compact by Prokhorov's theorem.
If $\pi_n\Rightarrow\pi$, then
$\pi_n\otimes\pi_n\Rightarrow\pi\otimes\pi$; since $\Gamma$
is continuous and nonnegative, Portmanteau's theorem yields
\[
\cE_\Gamma(\pi)
\le \liminf_{n\to\infty}\cE_\Gamma(\pi_n).
\]
Furthermore, superlinearity of $\varphi$ implies that its
convex conjugate
$\psi(a):=\sup_{t\ge0}\{at-\varphi(t)\}$ is finite
for every $a\in\mathbb R$. Consequently,
$\mathsf{D}_\varphi(\,\cdot\,,\mu\otimes\nu)$ is weakly lower
semicontinuous by
\cite[Corollary~51]{birrell2022fgamma}.

Both objectives are therefore nonnegative and weakly lower
semicontinuous on the weakly compact set $\Pi(\mu,\nu)$,
and are finite at $\mu\otimes\nu$. They attain their minima,
and their minimizer sets are weakly closed, hence weakly
compact. Finally, every regularized minimizer satisfies
\[
\varepsilon \mathsf{D}_\varphi(\pi^*,\mu\otimes\nu)
\le \mathcal F_\varepsilon^{\mu,\nu}(\pi^*)
\le \mathcal F_\varepsilon^{\mu,\nu}(\mu\otimes\nu)
=\cE_\Gamma(\mu\otimes\nu)<\infty,
\]
which proves the divergence bound.
\end{proof}

\begin{proof}[Proof of \Cref{lemma:conjugate-regularity}]
Under \Cref{assumption:regularizaion}, for $s\ge1$ we have
$\varphi''(s)\ge[((\lambda_0+\lambda_1)\vee\lambda_2)(1+s)]^{-1}$, which implies
$\varphi'(s)\to\infty$ as $s\to\infty$ and hence $\varphi(s)/s\to\infty$
as $s\to\infty$.
Denote $\ell:=\lim_{r\downarrow0}\varphi'(r)<1$.
Standard convex conjugacy \cite{rockafellar1998variational} therefore yields a finite
$\psi\in\cC^1(\bR)$ with
\begin{equation}
\label{eq:inverse-formula}
\psi'(s)=
\begin{cases}
0, & s\le\ell,\\
(\varphi')^{-1}(s), & s>\ell.
\end{cases}
\end{equation}
In particular, $\psi'\ge0$, $\psi'(1)=1$, and
$\psi'(s)\to\infty$ as $s\to\infty$.
The inverse-function theorem gives
\begin{equation}
\label{eq:inverse-second-derivative}
\psi''(s)=\frac{1}{\varphi''(\psi'(s))}>0,
\qquad s>\ell,
\end{equation}
with $\psi''$ continuous there.
Since $\ell<1$, $\psi$ is strictly convex on
$[1-\delta,\infty)$ for some $\delta>0$.
Let $I$ be a compact interval and set $R:=1+\max_I\psi'$.
By \eqref{eq:Aphi}, for $0<r\le R$,
\begin{align*}
\varphi''(r)
&\ge
\frac{r^\beta}
{\lambda_0+\lambda_1r^\beta+\lambda_2r^{\beta+1}}
%\ge\frac{r^\beta}{(\lambda_0+\lambda_1+\lambda_2r)(1+r)^\beta}
\ge
\frac{r^\beta}{((\lambda_0+\lambda_1)\vee\lambda_2)(1+R)^{\beta+1}}.
\end{align*}
For $s<t$ in $I$, put $a:=\psi'(s)$ and $b:=\psi'(t)$.
The case $a=b$ is immediate.
If $0<a<b$, \eqref{eq:inverse-formula} gives
\[
t-s=\varphi'(b)-\varphi'(a)
=\int_a^b\varphi''(r)\,dr.
\]
If $a=0<b$, then $\ell$ is finite, $s\le\ell$, and
\[
\int_0^b\varphi''(r)\,dr
=
\lim_{\eta\downarrow0}
\bigl(\varphi'(b)-\varphi'(\eta)\bigr)
=t-\ell\le t-s.
\]
Thus, in both cases,
\[
t-s %\geq \varphi'(b) - \varphi'(a)
\ge\int_a^b\varphi''(r)\,dr
\geq
\frac{(b-a)^{\beta+1}}
{(\beta+1)((\lambda_0+\lambda_1)\vee\lambda_2)(1+R)^{\beta+1}}.
\]
This proves \eqref{eq:holder-psi}.
Finally, if $\sup I<\varphi'(\theta)$, the assertion is vacuous.
Otherwise, since $\varphi'(\theta/2)>\ell$ and $\psi''$ is continuous
and positive on $(\ell,\infty)$, set
\[
m:=\min_{u\in[\varphi'(\theta/2),\,\sup I]}\psi''(u)>0,
\qquad
\rho:=\min\left\{1,
\frac{\varphi'(\theta)-\varphi'(\theta/2)}
{1+\operatorname{diam}I}\right\}\in(0,1].
\]
For $s,t\in I$ with $\psi'(t)\ge\theta$, we have
$t\ge\varphi'(\theta)$ and
\[
\varphi'(\theta/2)
\le t-\rho\operatorname{diam}I
\le t+\rho(s-t)\le\sup I.
\]
Thus both $t$ and $t+\rho(s-t)$ belong to
$[\varphi'(\theta/2),\sup I]$.
Strong convexity on this interval and convexity of
$B_\psi(\cdot,t)$, with $B_\psi(t,t)=0$, yield
\[
\frac{m\rho^2}{2}|s-t|^2
\le B_\psi(t+\rho(s-t),t)
\le \rho B_\psi(s,t).
\]
Hence \eqref{eq:active-curvature} holds with
$\kappa_{I,\theta}=m\rho/2>0$.
\end{proof}

\begin{proof}[Proof of \Cref{lemma:GW-linearization}]
By the moment assumption,
$c_*$ is well-defined and integrable with respect to every
$\eta\in\Pi(\mu,\nu)$. 
For $\theta\in\Pi(\mu,\nu)$, suppose without loss of generality $\mathsf{D}_\varphi(\theta,\mu\otimes\nu)<\infty$
and set
$\pi_t:=(1-t)\pi^*+t\theta$ for $0<t<1$.
Symmetry of $\Gamma$ gives
\[
\cE_\Gamma(\pi_t)-\cE_\Gamma(\pi^*)
=
t\int c_*\,d(\theta-\pi^*)
+t^2\cE_\Gamma(\theta-\pi^*).
\]
Moreover, convexity yields
\[
\mathsf{D}_\varphi(\pi_t,\mu\otimes\nu)
-
\mathsf{D}_\varphi
(\pi^*,\mu\otimes\nu)
\le
t
\bigl(
\mathsf{D}_\varphi(\theta,\mu\otimes\nu)-\mathsf{D}_\varphi(\pi^*,\mu\otimes\nu)
\bigr).
\]
Combining with optimality of $\pi^*$, we obtain
\[
\begin{aligned}
0
&\le
\cE_\Gamma(\pi_t)-\cE_\Gamma(\pi^*)
+\varepsilon\bigl(\mathsf{D}_\varphi(\pi_t,\mu\otimes\nu)-\mathsf{D}_\varphi(\pi^*,\mu\otimes\nu)\bigr)\\
&\le
t\left[
\int c_*\,d(\theta-\pi^*)
+\varepsilon\bigl(\mathsf{D}_\varphi(\theta,\mu\otimes\nu)-\mathsf{D}_\varphi(\pi^*,\mu\otimes\nu)\bigr)
\right]
+t^2\cE_\Gamma(\theta-\pi^*).
\end{aligned}
\]
Dividing by $t$ and letting $t\downarrow0$ gives
\[
\int c_*\,d\theta
+
\varepsilon\, \mathsf{D}_\varphi(\theta,\mu\otimes\nu)
\ge
\int c_*\,d\pi^*
+
\varepsilon \,\mathsf{D}_\varphi(\pi^*,\mu\otimes\nu).
\]
Then \eqref{eq:GW-linearized-OT} follows.
\end{proof}

\begin{proof}[Proof of \Cref{proposition:Agamma}]

We write $\Gamma_{r,s}=\Gamma$ and adopt the convention $1/p=0$ for $p=\infty$.
For $t>0$, set
\[
\begin{aligned}
T_t
:=&
\iint_{\cX\times\cX}d_{\cX}(x_1,x_2)^t
\,d(\mu\otimes\mu)(x_1,x_2)
+
\iint_{\cX\times\cX}d_{\cX}(x_1,x_2)^t
\,d (\tilde\mu\otimes\tilde\mu)(x_1,x_2)\\
&+
\iint_{\cY\times\cY}d_{\cY}(y_1,y_2)^t
\,d(\nu\otimes\nu)(y_1,y_2)
+
\iint_{\cY\times\cY}d_{\cY}(y_1,y_2)^t
\,d(\tilde\nu\otimes\tilde\nu)(y_1,y_2).
\end{aligned}
\]
\Cref{assumption:Agamma} implies $T_{rs}<\infty$.
Moreover,
\[
0\le
\Gamma((x,y),(x',y'))
\le
d_{\cX}(x,x')^{rs}+d_{\cY}(y,y')^{rs}.
\]
Hence, for every $\pi\in\Pi(\mu,\nu)$ and
$\widetilde\pi\in\Pi(\tilde\mu,\tilde\nu)$,
\[
0\le
\cE_{\Gamma}(\pi)+\cE_{\Gamma}(\widetilde\pi)
\le T_{rs}<\infty.
\]
Fix such $\pi$ and $\widetilde\pi$.
If $W_p(\pi,\widetilde\pi)=\infty$,
\eqref{eq:Agamma} holds for every $\Lambda>0$.
We therefore assume that $W_p(\pi,\widetilde\pi)<\infty$
and take any $\gamma\in\Pi(\pi,\widetilde\pi)$ satisfying
\[
\|d_{\cX\times\cY,p}\|_{L^p(\gamma)}<\infty.
\]
Set
\[
M:=
\max\{
d_{\cX}(x,x'),d_{\cY}(y,y'),
d_{\cX}(u,u'),d_{\cY}(v,v')
\}.
\]
If $M=0$, both kernel values in the following comparison
vanish. For $M>0$, we have
\begin{align}
\label{assumption_cross_verification_inequality}
&\left|
\Gamma\bigl((x,y),(x',y')\bigr)
-
\Gamma\bigl((u,v),(u',v')\bigr)
\right|
\nonumber\\
&\le
rM^{s(r-1)}
\left|
\left|d_{\cX}(x,x')^s-d_{\cY}(y,y')^s\right|
-
\left|d_{\cX}(u,u')^s-d_{\cY}(v,v')^s\right|
\right|
\nonumber\\
&\le
rM^{s(r-1)}
\Bigl(
\left|d_{\cX}(x,x')^s-d_{\cX}(u,u')^s\right|
+
\left|d_{\cY}(y,y')^s-d_{\cY}(v,v')^s\right|
\Bigr)
\nonumber\\
&\le
rsM^{rs-1}
\Bigl(
\left|d_{\cX}(x,x')-d_{\cX}(u,u')\right|
+
\left|d_{\cY}(y,y')-d_{\cY}(v,v')\right|
\Bigr)
\nonumber\\
&\le
rsM^{rs-1}
\Bigl(
d_{\cX}(x,u)+d_{\cX}(x',u')
+d_{\cY}(y,v)+d_{\cY}(y',v')
\Bigr)
\nonumber\\
&\le
rs\,2^{1-1/p}M^{rs-1}
\Bigl(
d_{\cX\times\cY,p}((x,y),(u,v))
+
d_{\cX\times\cY,p}((x',y'),(u',v'))
\Bigr).
\end{align}
Here, the first and third inequalities use
$|a^q-b^q|
\le
q(a\vee b)^{q-1}|a-b|$
for $q\ge1$ and $a,b\ge0$,
which follows from the mean-value theorem. The second and fourth inequalities follow from the triangle inequality. The last follows from H\"older's inequality
$a+b\le 2^{1-1/p}(a^p+b^p)^{1/p}$
for $a,b\ge0$ under  $1\leq p< +\infty$ and  $a+b\le2(a\vee b)$ under  $p=\infty$.
Write $z=(x,y)$, $w=(u,v)$, $z'=(x',y')$, and $w'=(u',v')$.
Since $\gamma\in\Pi(\pi,\widetilde\pi)$, the triangle inequality gives
\[
\bigl|\mathcal{E}_{\Gamma}(\pi)
      -\mathcal{E}_{\Gamma}(\widetilde\pi)\bigr|
\le
\iint \bigl|\Gamma(z,z')-\Gamma(w,w')\bigr|
\,d\gamma(z,w)\,d\gamma(z',w').
\]

Suppose first that $rs=1$.
Integrating
\eqref{assumption_cross_verification_inequality}
then using the $L^p$ triangle inequality 
and taking infimum over $\gamma\in\Pi(\pi,\widetilde\pi)$ gives, for every $1\le p\le\infty$,
\[
|\cE_{\Gamma}(\pi)-\cE_{\Gamma}(\widetilde\pi)|
\le
2^{2-1/p}
\inf_{\gamma\in\Pi(\pi,\widetilde\pi)}
\|d_{\cX\times\cY,p}\|_{L^p(\gamma)}
=
2^{2-1/p}
W_p(\pi,\widetilde\pi)
\]

Then we assume that $rs>1$ and $1<p<\infty$.
Set $\ell:=(rs-1)\frac{p}{p-1}$.
\Cref{assumption:Agamma} implies $T_\ell<\infty$.
We have
\[
\|M^{rs-1}\|_{L^{\frac{p}{p-1}}(\gamma\otimes\gamma)}
=
\left(
\iint M^\ell\,d(\gamma\otimes\gamma)
\right)^{\frac{p-1}{p}}
\le
T_\ell^{\frac{p-1}{p}}.
\]
Applying H\"older's inequality to
\eqref{assumption_cross_verification_inequality},
then using the triangle inequality,
and taking infimum over $\gamma\in\Pi(\pi,\widetilde\pi)$ yields
\begin{align*}
|\cE_{\Gamma}(\pi)-\cE_{\Gamma}(\widetilde\pi)|
%&\le rs\,2^{1-1/p}T_\ell^{\frac{p-1}{p}}\inf_{\gamma\in\Pi(\pi,\widetilde\pi)}\Bigl\|d_{\cX\times\cY,p}((x,y),(u,v))+d_{\cX\times\cY,p((x',y'),(u',v'))\Bigr\|_{L^p(\gamma\otimes\gamma)}\\
&\le
2^{2-1/p}rs\,T_\ell^{\frac{p-1}{p}}
\inf_{\gamma\in\Pi(\pi,\widetilde\pi)} \|d_{\cX\times\cY,p}\|_{L^p(\gamma)}
=
2^{2-1/p}rs\,T_\ell^{\frac{p-1}{p}}
W_p(\pi,\widetilde\pi).
\end{align*}

Next, assume that $rs>1$ and $p=\infty$.
\Cref{assumption:Agamma} indicates  $T_{rs-1}<\infty$, and
\[
\|M^{rs-1}\|_{L^1(\gamma\otimes\gamma)}
\le T_{rs-1}.
\]
Analogously, one can obtain
\begin{align*}
|\cE_{\Gamma}(\pi)-\cE_{\Gamma}(\widetilde\pi)|
%&\le2rs\,T_{rs-1}\inf_{\gamma\in\Pi(\pi,\widetilde\pi)}\Bigl\|d_{\cX\times\cY,\infty}((x,y),(u,v))+d_{\cX\times\cY,\infty}((x',y'),(u',v'))\Bigr\|_{L^\infty(\gamma\otimes\gamma)}\\
&\le
4rs\,T_{rs-1}
\inf_{\gamma\in\Pi(\pi,\widetilde\pi)}\|d_{\cX\times\cY,\infty}\|_{L^\infty(\gamma)}
=
4rs\,T_{rs-1}
W_p(\pi,\widetilde\pi).
\end{align*}

Finally, assume $rs>1$ and $p=1$.
By \Cref{assumption:Agamma}, all four marginals have bounded support. Hence there exists $D_0>0$ such that
$\operatorname{supp}\mu\cup\operatorname{supp}\tilde\mu$
and
$\operatorname{supp}\nu\cup\operatorname{supp}\tilde\nu$
have diameters at most $D_0$, for some $D_0>0$ 
and then $M\le D_0$.
As a consequence, we deduce analogously
\[
|\cE_{\Gamma}(\pi)-\cE_{\Gamma}(\widetilde\pi)|
\le
2rsD_0^{rs-1}
\inf_{\gamma\in\Pi(\pi,\widetilde\pi)}
\|d_{\cX\times\cY,p}\|_{L^1(\gamma)}
=
2rsD_0^{rs-1} W_1(\pi,\widetilde\pi).
\]
This completes the proof. Namely, one could take 
\[
\Lambda=
\begin{cases}
2^{2-1/p},
& rs=1,\\[0.6ex]
2^{2-1/p}rs\,
T_{(rs-1)p/(p-1)}^{(p-1)/p},
& rs>1,\quad 1<p<\infty,\\[0.6ex]
4rs\,T_{rs-1},
& rs>1,\quad p=\infty,\\[0.6ex]
2rs\,D_0^{rs-1},
& rs>1,\quad p=1.
\end{cases}
\]
\end{proof}

\begin{proof}[Verification details of \Cref{remark:Cgamma-verification}]
Fix $\theta\in\Pi(\mu,\nu)$ with
$\cF_\varepsilon^{\mu,\nu}(\theta)<\infty$, and write
$
f:=\frac{d\theta}{d(\mu\otimes\nu)},
f_*:=\frac{d\pi^*}{d(\mu\otimes\nu)}.
$   
Since $a+b-2\leq ab$ for $a,b\geq 1$,
\begin{align}\label{eqn:verification_bound_distortion}
0\le\Gamma_{r,s}((x,y),(x',y'))
&\le
d_{\cX}(x,x')^{rs}+d_{\cY}(y,y')^{rs}\\
&\le
2^{rs-1}\bigl(w(x,y)+w(x',y')-2\bigr)
\le
2^{rs-1}w(x,y)w(x',y').\nonumber
\end{align}
Under the standing moment assumptions,
$\int w\,d\eta<\infty$ for every $\eta\in\Pi(\mu,\nu)$,
so this bound ensures finiteness of all upcoming integrals. In particular, we have
\begin{equation}
\label{eq:Cgamma-weighted-energy}
\cE_\Gamma(\theta-\pi^*)
\ge
-2^{rs-2}\|w(\theta-\pi^*)\|_{\mathrm{TV}}^2.
\end{equation} 
Actually, for $\sigma:=\theta-\pi^*$, the inequality
$|\sigma|\le\theta+\pi^*$, symmetry and nonnegativity of $\Gamma$
and \eqref{inequality_gamma_bound_assumption_verificaion}, \eqref{eqn:verification_bound_distortion} 
imply
\begin{align*}
&\cE_\Gamma(\theta-\pi^*)
=
\cE_\Gamma(\sigma^+)+\cE_\Gamma(\sigma^-)
-2\iint\Gamma(z,z')\,d\sigma^+(z)\,d\sigma^-(z')
\notag\\
&\ge
-2^{rs}
\left(\int w\,d\sigma^+\right)
\left(\int w\,d\sigma^-\right)
\ge
-2^{rs-2}\left(\int w\,d|\sigma|\right)^2
=
-2^{rs-2}\|w(\theta-\pi^*)\|_{\mathrm{TV}}^2,
\end{align*}
where $\sigma=\sigma^+-\sigma^-$ is the Jordan decomposition and the last inequality uses $4ab\le(a+b)^2$.
For $0<t<1$, set $f_t:=(1-t)f_*+tf$.
For the regularizers considered below, the integral Taylor formula yields
\begin{equation}
\label{eq:Cgamma-bregman-integral}
B^{\mu,\nu}_\varphi(\theta\mid\pi^*)
=
\int_0^1(1-t)
\int \varphi''(f_t)(f-f_*)^2\,d(\mu\otimes\nu)\,dt.
\end{equation}

\begin{comment}
Whenever $\mu$ and $\nu$ have compact supports, set
\[
M:=
\max\left\{
\operatorname{diam}(\operatorname{supp}\mu)^{rs},
\operatorname{diam}(\operatorname{supp}\nu)^{rs}
\right\}<\infty.
\]
For $(x,y)$ and $(x',y')$ in
$\operatorname{supp}\mu\times\operatorname{supp}\nu$,
\[
0\le\Gamma_{r,s}((x,y),(x',y'))
=
\left|d_{\cX}(x,x')^s-d_{\cY}(y,y')^s\right|^r
\le
\max\left\{
d_{\cX}(x,x')^{rs},d_{\cY}(y,y')^{rs}
\right\}
\le M.
\]
Hence, for $\sigma:=\theta-\pi^*$, boundedness of
$\Gamma$ gives
\[
\iint\Gamma(z,z')\,d|\sigma|(z)\,d|\sigma|(z')
\le M\|\sigma\|_{\mathrm{TV}}^2<\infty.
\]
Moreover, since $\sigma(\cX\times\cY)=0$, its Jordan decomposition
$\sigma=\sigma^+-\sigma^-$ satisfies
$\sigma^+(\cX\times\cY)
=
\sigma^-(\cX\times\cY)
=
\frac12\|\sigma\|_{\mathrm{TV}}$.
Therefore, symmetry and nonnegativity of $\Gamma$ imply
\begin{align}
\cE_\Gamma(\theta-\pi^*)
&=
\cE_\Gamma(\sigma^+)+\cE_\Gamma(\sigma^-)
-2\iint\Gamma(z,z')\,d\sigma^+(z)\,d\sigma^-(z')\nonumber
\\
&\ge
-2M\,\sigma^+(\cX\times\cY)\sigma^-(\cX\times\cY)
=
-\frac M2\|\theta-\pi^*\|_{\mathrm{TV}}^2.
\label{eq:Cgamma-bounded-energy}
\end{align}
\end{comment}

\medskip
\noindent
\textup{(i) Entropy regularization.}
For $\varphi(t)=t\log t$, an admissible constant is
\[
\kappa=2^{rs-1}S.
\]
Indeed, since $\varphi''(t)=t^{-1}$, Cauchy--Schwarz inequality gives
\[
\begin{aligned}
\|w(\theta-\pi^*)\|_{\mathrm{TV}}^2
&=
\left(\int w|f-f_*|\,d(\mu\otimes\nu)\right)^2\\
&\le
\left(\int\frac{(f-f_*)^2}{f_t}\,d(\mu\otimes\nu)\right)
\left(\int w^2f_t\,d(\mu\otimes\nu)\right)\\
&\le
S\int\frac{(f-f_*)^2}{f_t}\,d(\mu\otimes\nu).
\end{aligned}
\]
Fubini's theorem then yields
\begin{align*}
\frac{\|w(\theta-\pi^*)\|_{\mathrm{TV}}^2}{2}
&=
\int_0^1(1-t)\|w(\theta-\pi^*)\|_{\mathrm{TV}}^2\,dt
\le
S\int_0^1(1-t)
\int\frac{(f-f_*)^2}{f_t}\,d (\mu\otimes\nu)\,dt\\
&=
S\int
\left[
\int_0^1(1-t)\frac{(f-f_*)^2}{f_t}\,dt
\right]d(\mu\otimes\nu)
=
S\int b_\varphi(f\mid f_*)\,d(\mu\otimes\nu)\\
&=
S B^{\mu,\nu}_\varphi(\theta\mid\pi^*).
\end{align*}
Combining this with \eqref{eq:Cgamma-weighted-energy} proves the claim.

\begin{comment}
On compact supports, the same argument with weight $1$ gives
$B^{\mu,\nu}_\varphi(\theta\mid\pi^*)
\ge\frac12\|\theta-\pi^*\|_{\mathrm{TV}}^2.$
Thus \eqref{eq:Cgamma-bounded-energy} gives the sharper choice
\[
\kappa=M, \qquad \varepsilon>M.
\]
\end{comment}

\medskip
\noindent
\textup{(ii) Tsallis regularization with $1<\alpha\le2$.}
For $\varphi(t)=(t^\alpha-1)/\alpha$, an admissible constant is
\[
\kappa=
\frac{2^{rs-1}}{\alpha-1}
\left(\int w^2\,d(\mu\otimes\nu)\right)^{\alpha-1}S^{2-\alpha}.
\]
Since
$f_t(\mu\otimes\nu)=(1-t)\pi^*+t\theta\in\Pi(\mu,\nu)$,
the definition of $S$ gives $\int w^2f_t\,d(\mu\otimes\nu)\le S$.
Using H\"{o}lder inequality twice now gives
\begin{align*}
\|w(\theta-\pi^*&)\|_{\mathrm{TV}}^2
%&=\left(\int_{\{f_t>0\}}\bigl(|f-f_*|f_t^{(\alpha-2)/2}\bigr)\bigl(wf_t^{(2-\alpha)/2}\bigr)\,d(\mu\otimes\nu)\right)^2\\
\le
\left(
\int (f-f_*)^2f_t^{\,\alpha-2}\,d(\mu\otimes\nu)
\right)
\left(
\int w^2f_t^{\,2-\alpha}\,d(\mu\otimes\nu)
\right)\\
&\le\left(\int w^2\,d(\mu\otimes\nu)\right)^{\alpha-1}
\left(\int w^2f_t\,d(\mu\otimes\nu)\right)^{2-\alpha}\int (f-f_*)^2f_t^{\,\alpha-2}\,d(\mu\otimes\nu)\\
&\le
\left(\int w^2\,d(\mu\otimes\nu)\right)^{\alpha-1}
S^{2-\alpha}
\int (f-f_*)^2f_t^{\,\alpha-2}\,d(\mu\otimes\nu).
\end{align*}
Since $\varphi''(u)=(\alpha-1)u^{\alpha-2}$,
the integral Taylor formula
\eqref{eq:Cgamma-bregman-integral} and Fubini's theorem yield
\begin{align*}
B^{\mu,\nu}_{\varphi}(\theta\mid\pi^*)
&=
(\alpha-1)\int_0^1(1-t)
\int (f-f_*)^2f_t^{\,\alpha-2}
\,d(\mu\otimes\nu)\,dt\\
&\ge
\frac{(\alpha-1)\|w(\theta-\pi^*)\|_{\mathrm{TV}}^2}
{\left(\int w^2\,d(\mu\otimes\nu)\right)^{\alpha-1}
S^{2-\alpha}}
\int_0^1(1-t)\,dt
=
\frac{(\alpha-1) \|w(\theta-\pi^*)\|_{\mathrm{TV}}^2}
{2\left(\int w^2\,d(\mu\otimes\nu)\right)^{\alpha-1}
S^{2-\alpha}}
.
\end{align*}
Combining this with \eqref{eq:Cgamma-weighted-energy} proves the claim.

\begin{comment}
On compact supports, using weight $1$ instead gives
$B_{\varphi_\alpha}(\theta\mid\pi^*)
\ge
\frac{\alpha-1}{2}\|\theta-\pi^*\|_{\mathrm{TV}}^2$.
Hence one may take
\[
\kappa=\frac{M}{\alpha-1},
\qquad
\varepsilon>\frac{M}{\alpha-1}.
\]
\end{comment}

\medskip
\noindent
\medskip
\noindent
\textup{(iii) Tsallis regularization with $\alpha>2$ on compact supports.}
Suppose that $\mu$ and $\nu$ have compact supports, and set
\[
M:=
\max\left\{
\operatorname{diam}(\operatorname{supp}\mu)^{rs},
\operatorname{diam}(\operatorname{supp}\nu)^{rs}
\right\}.
\]
If $\varepsilon\ge5M$, then
\Cref{assumption:Cgamma} holds on the prescribed coupling domain with
\[
\kappa:=
\frac{\alpha M}{2(\alpha-1)}
\left(1-\frac{4M}{\varepsilon}\right)^{-(\alpha-2)/(\alpha-1)}
<\varepsilon.
\]
We first use
\Cref{lemma:GW-linearization} to regard $\pi^*$ as an optimizer
of an auxiliary regularized OT problem.
Specifically, with
\[
c_*(x,y):=2\int\Gamma((x,y),(x',y'))\,d\pi^*(x',y'),
\]
we have $0\le c_*\le2M$ and
\[
\pi^*\in
\operatorname*{argmin}_{\theta\in\Pi(\mu,\nu)}
\left\{
\int c_*\,d\theta
+\varepsilon \,\mathsf D_{\varphi}(\theta,\mu\otimes\nu)
\right\}.
\]
Bounded-cost convex-regularized transport duality
\cite{alberto.gonzalez-marcel.nutz_austin.stromme2026}
provides bounded potentials $u,v$ such that
$$f_*(x,y)=
\left(
\frac{u(x)+v(y)-c_*(x,y)}{\varepsilon}
\right)_+^{1/(\alpha-1)}$$
with 
$\operatorname{ess\,osc}_{\nu}v\le2M$
and then
$
\operatorname{ess\,osc}_{y\sim\nu}
\bigl(u(x)+v(y)-c_*(x,y)\bigr)\le4M.
$
Equation $\int f_*(x,y) d\nu(y)=1$ for
$f_*$ implies
$
\operatorname*{ess\,sup}_{y\sim\nu}
\bigl(u(x)+v(y)-c_*(x,y)\bigr)\ge\varepsilon.
$
Thus $u(x)+v(y)-c_*(x,y)\ge\varepsilon-4M$  and hence
\begin{equation}
\label{eq:Cgamma-Tsallis-density}
f_*\ge
\Bigl(1-\frac{4M}{\varepsilon}\Bigr)^{\frac{1}{\alpha-1}}
\end{equation}
Note that $f_t \geq (1-t)f_*$,  using \eqref{eq:Cgamma-Tsallis-density}, the integral Taylor formula \eqref{eq:Cgamma-bregman-integral} and Fubini's theorem yield
\begin{align*}
B^{\mu,\nu}_{\varphi}&(\theta\mid\pi^*)
\geq
(\alpha-1)
\Bigl(1-\frac{4M}{\varepsilon}\Bigr)^{\frac{\alpha-2}{\alpha-1}}
\int_0^1
(1-t)^{\alpha-1} \int  (f-f_*)^{2}
\,d(\mu\otimes\nu)\,dt\\
&=
\frac{\alpha-1}{\alpha}
\Bigl(1-\frac{4M}{\varepsilon}\Bigr)^{\frac{\alpha-2}{\alpha-1}}
\int(f-f_*)^2\,d(\mu\otimes\nu)
\ge
\frac{\alpha-1}{\alpha}
\Bigl(1-\frac{4M}{\varepsilon}\Bigr)^{\frac{\alpha-2}{\alpha-1}}
\|\theta-\pi^*\|_{\mathrm{TV}}^2.
\end{align*}
We apply Cauchy--Schwarz with $\mu\otimes\nu$ for the last inequality.
Since $0\le\Gamma\le M$, the Jordan decomposition as above gives
\[
\cE_\Gamma(\theta-\pi^*)
\ge-\frac M2\|\theta-\pi^*\|_{\mathrm{TV}}^2
\ge-\kappa B^{\mu,\nu}_{\varphi}(\theta\mid\pi^*).
\]
Finally, $\varepsilon\ge5M$ implies
\[
\frac{\kappa}{\varepsilon}
\le
\frac{\alpha}{2(\alpha-1)}\,5^{-1/(\alpha-1)}
<1,
\]
which completes the verification.
\end{proof}

\begin{proof}[Proof of \Cref{lem:gw-cost-bounds}]
Since $\|A\|_{\mathrm{op}}\le\|A\|_F\le R_xR_y/2$,
\[
 |c_A(x,y)|
 \le4R_x^2R_y^2+32R_xR_y\|A\|_F
 \le20R_x^2R_y^2.
\]
For $A,A'\in D$,
\[
 |c_A(x,y)-c_{A'}(x,y)|
 =32|x^\top(A-A')y|
 \le32R_xR_y\|A-A'\|_F.
\]
Moreover, $|\|x\|^2-\|x'\|^2|\le2R_x\|x-x'\|$, so
\begin{align*}
 |c_A(x,y)-c_A(x',y)|
 &\le4|\|x\|^2-\|x'\|^2|R_y^2
       +32\|x-x'\|\|A\|_{\mathrm{op}}R_y\\
 &\le(8R_xR_y^2+16R_xR_y^2)\|x-x'\|=24R_xR_y^2\|x-x'\|.
\end{align*}
The other coordinate bound follows by interchanging $x$ and $y$.
\end{proof}

\begin{proof}[Proof of \Cref{lem:covering}]
Throughout the proof, $C$ denotes a finite positive constant depending on  $R_x,d_x,R_y,d_y,\varphi,\varepsilon$.
Its value may change from one display to the next.

We first establish the following claim: there exist a finite
integer $J\ge1$ satisfying \eqref{eq:covering-J-bound} and fixed
continuous linear functionals
$\tau_j:\mathcal C(\mathcal X)\to\bR$, $j=1,\ldots,J$, satisfying
\[
 |\tau_j(h)|\le\|h\|_\infty
 \qquad\text{for every }h\in\mathcal C(\mathcal X),
\]
such that for every $(A,f,g),(A',f',g')\in\cS$,
\begin{equation}
 \|H_{A,f,g}-H_{A',f',g'}\|_\infty
 \le C\bigl(
 \|A-A'\|_F^{\frac{1}{2(1+\beta)}}
 +\|\tau(f)-\tau(f')\|_{\ell^\infty_J}^{\frac{1}{1+\beta}}
 \bigr),
\label{eq:integrand-stability}
\end{equation}
where $\tau(f)=(\tau_1(f),\ldots,\tau_J(f))$, $\|a\|_{\ell^\infty_J}=\max_{1\le j\le J}|a_j|, a\in\bR^J$.
Assuming this claim, define
\[
 \Theta
 :=
 \{(A,\tau(f)):(A,f,g)\in\cS\}
 \subset\bR^{d_xd_y+J}.
\]
This set is bounded because $\mathcal D$ is bounded and $|\tau_j(f)|\le B$ under the normalization. 
Here, we identify \(\mathbb R^{d_xd_y+J}\) with \(\mathbb R^{d_x\times d_y}\times\mathbb R^J\) and equip it with the Euclidean norm
$
|(A,a)|
:=
\left(\|A\|_F^2+\sum_{j=1}^J |a_j|^2\right)^{1/2}.
$
By \eqref{eq:integrand-stability}, $z\mapsto H_z$ is a well-defined map from $\Theta$ onto $\cH$.

Choose $R_\Theta\ge1$ such that
$\Theta\subset B(0,R_\Theta)$. On this bounded set,
\eqref{eq:integrand-stability} implies
\[
 \|H_z-H_{z'}\|_\infty
 \le L_\Theta|z-z'|^{\frac{1}{2(1+\beta)}}
\]
for some $L_\Theta\ge1$. 
A parameter net with centers in $\Theta$ therefore gives
\begin{align}
 N(\eta,\cH,\|\cdot\|_\infty)
 \le
 N\bigl((\eta/L_\Theta)^{2(1+\beta)},
        \Theta,|\cdot|\bigr)
 \le
 \left[
  1+2R_\Theta(L_\Theta/\eta)^{2(1+\beta)}
 \right]^{d_xd_y+J}.
 \label{eq:volumetric-cover}
\end{align}
The final inequality follows from a standard volumetric packing argument applied to a maximal separated subset of $\Theta$; see \cite[Section~4.2]{vershynin2019high}. 
Also note that
$\sup_{H\in\cH}\|H\|_\infty
\le2B+\varepsilon\sup_I|\psi|<\infty$.
Equation~\eqref{eq:entropy-bound} follows.
Thus, it remains to prove \eqref{eq:integrand-stability}.

\textit{Step 1: Partition.}
Every element $(A,f,g)$ belonging to \(\mathcal S\) satisfies \(\|f\|_\infty,\|g\|_\infty\le B\) under the prescribed normalization, and inherits the respective coordinatewise Lipschitz constants $24R_xR_y^2$ and $24R_x^2R_y$, respectively.
The interval $I$ contains all dual potential related arguments below.
Write $M=\max\{1,\sup_I\psi'\}$, and we retain $L,\kappa=\kappa_{I,1/2}$ as above.
Take $(A,f,g),(A',f',g')\in\cS$ and put
\[
 u=f'-f,\qquad v=g'-g,\qquad
 \xi=\frac{f\oplus g-c_A}{\varepsilon},
 \qquad \rho=\psi'(\xi).
\]
Then $0\le\rho\le M$. Define
\begin{equation*}
 Q(u,v)
 :=
 (\mu\otimes\nu)
 \bigl[\ind_{\{\rho\ge1/2\}}(u\oplus v)^2\bigr].
\end{equation*}
Fubini's theorem and the first-order condition \eqref{eqn:schrodinger_system} give
\begin{align*}
 (\mu\otimes\nu)[\rho(u\oplus v)]
 &=
 \int_{\mathcal X}
 u(x)
 \int_{\mathcal Y}\rho(x,y)\,d\nu(y)
d\mu(x)+
 \int_{\mathcal Y}
 v(y)
 \int_{\mathcal X}\rho(x,y)\,d\mu(x)
 d\nu(y)\\
 &=
 \int u(x)\,d\mu(x)+\int v(y)\,d\nu(y)
 =
 (\mu\otimes\nu)(u\oplus v).
\end{align*}
Consequently, writing $\Phi_A=\Phi_{\mu,\nu}^{c_A}$,
\begin{align*}
 \Phi_A(f,g)-\Phi_A(f',g')
 =
 \varepsilon(\mu\otimes\nu)
 B_\psi\!\left(
 \xi+\frac{u\oplus v}{\varepsilon},\xi
 \right)\ge
 \frac{\kappa}{\varepsilon}Q(u,v).
\end{align*}
The optimality of $(f',g')$ for $c_{A'}$ yields
\begin{align*}
 \Phi_A(f,g)-\Phi_A(f',g')
 %&=\Phi_A(f,g)-\Phi_{A'}(f,g)+\Phi_{A'}(f,g)-\Phi_{A'}(f',g') +\Phi_{A'}(f',g')-\Phi_A(f',g')
 \le
 2M\|c_A-c_{A'}\|_\infty,
\end{align*}
since
\[
 |\Phi_A(f,g)-\Phi_{A'}(f,g)|
 \le M\|c_A-c_{A'}\|_\infty,
 \qquad
 |\Phi_A(f',g')-\Phi_{A'}(f',g')|
 \le M\|c_A-c_{A'}\|_\infty.
\]
It follows that
\begin{equation}
 Q(u,v)
 \le
 \frac{2M\varepsilon}{\kappa}
 \|c_A-c_{A'}\|_\infty.
 \label{eq:active-gap-bound}
\end{equation}
For $a\in\{1/2,3/4\}$, set $m_a:=\frac{1-a}{M-a}$.
For any fixed $x\in\mathcal X$, the first-order condition \eqref{eqn:schrodinger_system} implies
\[
 1=\int\rho(x,y)\,d\nu(y)
 \le
 a\bigl(1-\nu\{\rho(x,\cdot)\ge a\}\bigr)
 +M\nu\{\rho(x,\cdot)\ge a\}.
\]
Hence
\begin{equation}
 \nu\{\rho(x,\cdot)\ge a\}\ge m_a,
 \qquad a\in\{1/2,3/4\}.
 \label{eq:active-mass_x}
\end{equation}
Analogously, for $y\in\mathcal Y$,
\begin{equation}
 \mu\{\rho(\cdot,y)\ge a\}\ge m_a,
 \qquad a\in\{1/2,3/4\}.
 \label{eq:active-mass_y}
\end{equation}
The common Lipschitz bounds \eqref{eq:gw-spatial-cost-bounds} and H\"older continuity
\eqref{eq:holder-psi} give
\begin{align} \label{eq:density-modulus}
 \sup_{y\in\mathcal Y}
 |\rho(x,y)-\rho(x',y)|
&\le L\varepsilon^{-\frac{1}{1+\beta}}
\sup_{y\in\mathcal Y}
\left|
f(x)-f(x')-c_A(x,y)+c_A(x',y)
\right|^{\frac{1}{1+\beta}}\nonumber\\
&\le L\varepsilon^{-\frac{1}{1+\beta}}
\sup_{y\in\mathcal Y}
 \left(
|f(x)-f(x')|
+|c_A(x,y)-c_A(x',y)|
\right)^{\frac{1}{1+\beta}}\\
&\le
 L\varepsilon^{-\frac{1}{1+\beta}}
 \bigl(48R_xR_y^2\|x-x'\|\bigr)^{\frac{1}{1+\beta}}.\nonumber
\end{align}

The right side of \eqref{eq:density-modulus} is at most
$1/4$ whenever $\|x-x'\|\le r_0:=
\frac{\varepsilon}
{48R_xR_y^2(4L)^{1+\beta}}$.
We hence choose a maximal $r_0/2$-separated set
$\{x_1,\ldots,x_{N_0}\}\subset\mathcal X$.
Its open balls of radius $r_0$ cover $\mathcal X$.
The balls $B(x_j,r_0/4)$ are pairwise disjoint and lie in
$B(0,R_x+r_0/4)\subset\bR^{d_x}$. Comparing their volumes gives
$ N_0\le
 \left(1+\frac{4R_x}{r_0}\right)^{d_x}.$
Form a disjoint Borel partition subordinate to the cover
$\{B(x_j,r_0)\}_{j=1}^{N_0}$ and discard zero-$\mu$-mass cells.
Relabel the remaining cells as $E_1,\ldots,E_J$, retaining their
associated centers $x_j$. Then
\[
 E_j\subset B(x_j,r_0),\qquad
 \mu(E_j)>0,\qquad
 \mu\biggl(\,\bigcup_{j=1}^J E_j\biggr)=1,\qquad
 1\le J\le N_0.
\]
This proves \eqref{eq:covering-J-bound}.
Based on this specific partition, define the fixed functionals
\begin{equation*}
\tau_j(h):=\frac1{\mu(E_j)}\int_{E_j}h\,d\mu,
 \qquad
 h\in\mathcal C(\mathcal X),\quad j=1,\ldots,J.
\end{equation*}
They are linear and continuous, with
$|\tau_j(h)|\le\|h\|_\infty$.
Set
$
 F_j=\{y\in\mathcal Y:\rho(x_j,y)\ge3/4\}.
$
By \eqref{eq:active-mass_x},
\begin{equation}\label{eqn:nu(Fj)_lower-bound}
\nu(F_j)\ge m_{3/4}.
\end{equation}
For $x\in E_j$ and $y\in F_j$,
\eqref{eq:density-modulus} gives
$\rho(x,y)\ge3/4-1/4=1/2$. Thus
\begin{equation}
 E_j\times F_j\subset\{\rho\ge1/2\}.
 \label{eq:active-rectangles}
\end{equation}

\textit{Step 2: $L^2$ perturbation bound.}
For every fixed $y\in\mathcal Y$ the exact variance decomposition
gives
\begin{align}\label{eqn:Ejdecomposition}
 \int_{E_j}[u(x)+v(y)]^2\,d\mu(x)
 =
 \int_{E_j}|u(x)-\tau_j(u)|^2\,d\mu(x)+
 \mu(E_j)[\tau_j(u)+v(y)]^2.
\end{align}
Integrating over $F_j$, summing over the disjoint cells, and using
\Cref{eqn:nu(Fj)_lower-bound,eq:active-rectangles,eqn:Ejdecomposition}, we obtain
\begin{align*}
 Q&(u,v)
 \ge
 \sum_{j=1}^J
 \int_{F_j}\int_{E_j}
 |u(x)+v(y)|^2\,d\mu(x)d\nu(y)\\
 &\ge
 \sum_{j=1}^J
 \int_{F_j}\int_{E_j}
 |u(x)-\tau_j(u)|^2\,d\mu(x)d\nu(y)\ge
 m_{3/4}\sum_{j=1}^J
 \int_{E_j}|u(x)-\tau_j(u)|^2\,d\mu(x).
\end{align*}
Because the cell masses sum to one,
\begin{align}
 \|u\|_{L^2(\mu)}^2
 =
 \sum_{j=1}^J\int_{E_j}|u(x)|^2\,d\mu(x)
 &\le
 \sum_{j=1}^J
 \int_{E_j}|u(x)-\tau_j(u)|^2\,d\mu(x)
 +\|\tau(f')-\tau(f)\|_{\ell^\infty_J}^2
 \nonumber\\
 &\le
 \frac1{m_{3/4}}Q(u,v)
 +
 \|\tau(u)\|_{\ell^\infty_J}^2.
 \label{eq:first-potential-l2}
\end{align}
For each $y\in\mathcal Y$, put
$G_y=\{x\in\mathcal X:\rho(x,y)\ge1/2\}$.
Equation~\eqref{eq:active-mass_y} yields
$\mu(G_y)\ge m_{1/2}$.
Integrating
$|v(y)|^2\le2|u(x)+v(y)|^2+2|u(x)|^2$
over $G_y$ gives
\[
 m_{1/2}|v(y)|^2
 \le
 2\int_{G_y}|u(x)+v(y)|^2\,d\mu(x)
 +2\int_{G_y}|u(x)|^2\,d\mu(x).
\]
Integrating over $y\in\mathcal Y$ and using Fubini's theorem gives
\begin{align} \label{eq:second-potential-l2}
 \|v\|_{L^2(\nu)}^2
 &\le
 \frac2{m_{1/2}}
 \biggl(
 \int_{\mathcal Y}\int_{G_y}
 |u(x)+v(y)|^2\,d\mu(x)d\nu(y)
 +
 \int_{\mathcal Y}\int_{G_y}
 |u(x)|^2\,d\mu(x)d\nu(y)
 \biggr)
 \nonumber\\
 &=
 \frac2{m_{1/2}}
 \biggl(
 Q(u,v)
 +
 \int_{\mathcal X}|u(x)|^2
 \nu\{y:\rho(x,y)\ge1/2\}\,d\mu(x)
 \biggr)\\
 &\le
 \frac2{m_{1/2}}
 \bigl(Q(u,v)+\|u\|_{L^2(\mu)}^2\bigr).\nonumber
\end{align}
Combining \eqref{eq:active-gap-bound},
\eqref{eq:first-potential-l2}, and
\eqref{eq:second-potential-l2} proves
\begin{equation}
 \|u\|_{L^2(\mu)}+\|v\|_{L^2(\nu)}
 \le
 C\bigl(\|c_A-c_{A'}\|_\infty^{1/2}
 +
 \|\tau(u)\|_{\ell^\infty_J}
 \bigr).
 \label{eq:both-potentials-l2}
\end{equation}

\textit{Step 3: From $L^2$ perturbation bound to uniform perturbation bound.}
The identity
$
 (\psi'(s)-\psi'(t))(s-t)
 =
 B_\psi(s,t)+B_\psi(t,s)
$
and convexity, together with \eqref{eq:active-curvature}, imply
\begin{equation}
 |\psi'(s)-\psi'(t)|\ge\kappa|s-t|
 \quad\text{if }s,t\in I
 \text{ and }\psi'(t)\ge1/2.
 \label{eq:inverse-slope}
\end{equation}
Since $\psi'$ is nondecreasing, the difference
$
 \psi'\!\left(
 \frac{f'(x)+g(y)-c_A(x,y)}{\varepsilon}
 \right)
 -
 \psi'\!\left(
 \frac{f(x)+g(y)-c_A(x,y)}{\varepsilon}
 \right)
$
has the same sign as $f'(x)-f(x)$.
Thus, \eqref{eq:inverse-slope} and the first-order condition \eqref{eqn:schrodinger_system} indicate that
\begin{align} \label{eq:scalar-inverse}
 &\quad
 \left|
 \int
 \psi'\!\left(
 \frac{f'(x)+g(y)-c_A(x,y)}{\varepsilon}
 \right)d\nu(y)-1
 \right|
 \nonumber\\
 &=
 \int
 \left|
 \psi'\!\left(
 \frac{f'(x)+g(y)-c_A(x,y)}{\varepsilon}
 \right)
 -
 \psi'\!\left(
 \frac{f(x)+g(y)-c_A(x,y)}{\varepsilon}
 \right)
 \right|d\nu(y)
\\
 &\ge
 \int_{\{y,\rho(x,y)\ge1/2\}}
 \left|
 \psi'\!\left(
 \frac{f'(x)+g(y)-c_A(x,y)}{\varepsilon}
 \right)
 -
 \psi'\!\left(
 \frac{f(x)+g(y)-c_A(x,y)}{\varepsilon}
 \right)
 \right|d\nu(y)\nonumber\\
 &\ge
 \frac{\kappa m_{1/2}}{\varepsilon}|u(x)|\nonumber.
\end{align}
At the same time, first-order condition \eqref{eqn:schrodinger_system} for $c_{A'}$ gives 
$
 \int
 \psi'\!\left(
 \frac{f'(x)+g'(y)-c_{A'}(x,y)}{\varepsilon}
 \right)d\nu(y)=1.
$
Using \eqref{eq:holder-psi}, one has
\begin{align*}
 \Bigl|
 \int
 \psi'\Bigl(
 \frac{f'(x)+g(y)-c_A(x,y)}{\varepsilon}
 \Bigr)d\nu(y)-1
 \Bigr|
 &\le
 \frac{L}{\varepsilon^{\frac{1}{1+\beta}}}
 \int
 (|v(y)|+\|c_A-c_{A'}\|_\infty)^{\frac{1}{1+\beta}}
 \,d\nu(y)\\
 &\le
 L\varepsilon^{-\frac{1}{1+\beta}}
 \bigl(
 \|v\|_{L^2(\nu)}^{\frac{1}{1+\beta}}
 +\|c_A-c_{A'}\|_\infty^{\frac{1}{1+\beta}}
 \bigr).
\end{align*}
We used $(r+s)^{\frac{1}{1+\beta}}\le r^{\frac{1}{1+\beta}}+s^{\frac{1}{1+\beta}}$ for $r,s\ge0$
and
$\int|v|^{\frac{1}{1+\beta}}\,d\nu
\le
\left(\int|v|^2\,d\nu\right)^{\frac{1}{2(1+\beta)}}
=
\|v\|_{L^2(\nu)}^{\frac{1}{1+\beta}}.
$
Together with \eqref{eq:scalar-inverse}, this yields
\begin{align}
 \|u\|_\infty
 &\le
 \frac{L\varepsilon^{\frac{\beta}{1+\beta}}}{\kappa m_{1/2}}
 \bigl(
 \|v\|_{L^2(\nu)}^{\frac{1}{1+\beta}}
 +\|c_A-c_{A'}\|_\infty^{\frac{1}{1+\beta}}
 \bigr).
 \label{eq:scalar-upgrade_u}
\end{align}
Analogously,
\begin{align}
 \|v\|_\infty
 &\le
 \frac{L\varepsilon^{\frac{\beta}{1+\beta}}}{\kappa m_{1/2}}
 \bigl(
 \|u\|_{L^2(\mu)}^{\frac{1}{1+\beta}}
 +\|c_A-c_{A'}\|_\infty^{\frac{1}{1+\beta}}
 \bigr).
 \label{eq:scalar-upgrade_v}
\end{align}
Since $\|c_A-c_{A'}\|_\infty\le 40 R_x^2 R_y^2$ by \Cref{lem:gw-cost-bounds}, equations
\eqref{eq:both-potentials-l2}--\eqref{eq:scalar-upgrade_v}
imply
\begin{equation}
 \|u\|_\infty+\|v\|_\infty
 \le
 C\bigl(
 \|c_A-c_{A'}\|_\infty^{\frac{1}{2(1+\beta)}}+\|\tau(u)\|_{\ell^\infty_J}^{\frac{1}{1+\beta}}
 \bigr).
 \label{eq:potential-uniform-stability}
\end{equation}
The function $\psi$ is $M$-Lipschitz on $I$. Hence,
\begin{align*}
 \|H_{A',f',g'}-H_{A,f,g}\|_\infty
 &\le
 (1+M)(\|u\|_\infty+\|v\|_\infty)
 +M\|c_A-c_{A'}\|_\infty\\
 &\le
 C\bigl(
 \|c_A-c_{A'}\|_\infty^{\frac{1}{2(1+\beta)}}+\|\tau(u)\|_{\ell^\infty_J}^{\frac{1}{1+\beta}}
 \bigr).
\end{align*}
Combining with Lemma \ref{lem:gw-cost-bounds}, this proves
\eqref{eq:integrand-stability}.
\end{proof}

\begin{proof}[Proof of \Cref{lem:empirical-reduction}]
We use the notations introduced in the proof of \Cref{thm:gw-duality}. For every $\pi\in\Pi(\lambda,\eta)$,
$\|B_\pi\|_F\le R_xR_y$, so $B_\pi/2\in \mathcal D$.
The minimum defining $V_\varepsilon(\lambda,\eta)$ exists by the same
compactness and continuity argument as in
Theorem~\ref{thm:gw-duality}. By \eqref{eq:distortionexpansion}, the completion of the square
used in \eqref{eq:gw-completed-square}, and $B_\pi/2\in \mathcal D$,
taking infima over $A$ and $\pi$ in either order gives
\[
\mathsf{GW}^1(\lambda,\eta)+V_\varepsilon(\lambda,\eta)
=
\inf_{\pi\in\Pi(\lambda,\eta)}
\left\{
\mathcal F_\varepsilon^{\lambda,\eta}(\pi)
-
\Bigl( 8a_\lambda^\top\int x\|y\|^2\,d\pi
+
8b_\eta^\top\int y\|x\|^2\,d\pi
\Bigr)
\right\}.
\]
Moreover, uniformly over $\pi\in\Pi(\lambda,\eta)$,
\[
\Bigl|
8a_\lambda^\top\int x\|y\|^2\,d\pi
+
8b_\eta^\top\int y\|x\|^2\,d\pi
\Bigr|
\le
8R_xR_y^2\|a_\lambda\|
+
8R_x^2R_y\|b_\eta\|.
\]
Comparison of the infima therefore yields
\[
\begin{aligned}
\left|
S_{2,2,\varepsilon}(\lambda,\eta)
-\mathsf{GW}^1(\lambda,\eta)
-V_\varepsilon(\lambda,\eta)
\right|
\le
8R_xR_y^2\|a_\lambda\|
+
8R_x^2R_y\|b_\eta\|.
\end{aligned}
\]
Apply this estimate with $(\lambda,\eta)=(\mu_n,\nu_n)$.
Since $\mu$ and $\nu$ are centered and
$M_{\mu,\nu}\le R_xR_y$, equation
\eqref{eq:gw-variational-duality}, with $M=R_xR_y$, gives
\[
S_{2,2,\varepsilon}(\mu,\nu)
=
\mathsf{GW}^1(\mu,\nu)+V_\varepsilon(\mu,\nu).
\]
The triangle inequality now proves
\eqref{eq:empirical-gw-reduction}.
\end{proof}

\end{document}